\documentclass[11pt,DIV=12,a4paper,oneside,headings=small]{scrartcl}
\usepackage[T1]{fontenc}
\usepackage[english]{babel}
\usepackage{amssymb,amsmath,amsfonts,amsthm,enumerate,dsfont,expdlist,tikz}
\usepackage[noblocks]{authblk}

\usepackage{graphicx,float}
\newcommand{\NN}{\mathbb{N}}	
\newcommand{\ZZ}{\mathbb{Z}}
\newcommand{\T}{\mathrm{T}}
\newcommand{\Dom}{\mathcal{D}}
\newcommand{\qUC}{\mathrm{qUC}}
\newcommand{\sfUC}{\mathrm{sfUC}}
\newcommand{\equi}{S}
\newcommand{\Cac}{F}
\newcommand{\Op}{\mathcal{H}}
\newcommand{\RE}{\operatorname{Re}}

\newcommand{\diver}{\operatorname{div}}
\newcommand{\supp}{\operatorname{supp}}
\newcommand{\Ran}{\operatorname{Ran}}
\renewcommand{\epsilon}{\varepsilon}
\newtheorem{theorem}{Theorem}
\newtheorem{lemma}[theorem]{Lemma}
\newtheorem{conjecture}[theorem]{Conjecture}
\theoremstyle{definition}
\newtheorem{definition}[theorem]{Definition}
\theoremstyle{remark}
\newtheorem{remark}[theorem]{Remark}
\begin{document}
%
%
%
%
%
%
\title{Scale-free quantitative unique continuation and equidistribution estimates for solutions of elliptic differential equations}
\author{Denis Borisov}
\affil{Institute of Mathematics, Ufa Scientific Center, Russian Academy of Sciences, 
Chernyshevskii str., 112, Ufa, Russia, 450077,
\& Bashkir State Pedagogical University,
October rev. st., 3a, Ufa, Russia, 450000,
\&
University of Hradec Kr\'alov\'e, Rokitansk\'eho 62, Hradec Kr\'alov\'e 50003,Czech Republic}
\author{Martin Tautenhahn}
\affil{Friedrich-Schiller-Universit\"at Jena, 07743 Jena, Germany \& Technische Universit\"at Chemnitz, 09126 Chemnitz, Germany}
\author{Ivan Veseli\'c}
\affil{Technische Universit\"at Dortmund, 44227 Dortmund, Germany}
\date{\vspace{-5ex}}
\maketitle
\begin{abstract}
We consider elliptic differential operators on either the 
entire
Euclidean space $\mathbb{R}^d$
or on subsets consisting of a cube $\Lambda_L$ of integer length $L$.
For eigenfunctions of the operator, and more general solutions of elliptic differential equations, we derive several quantitative unique continuation results.
The first result is  of local nature and estimates the vanishing order of a solution.
The second is a sampling result and compares the $L^2$-norm of a solution over a union of equidistributed $\delta$-balls in space with the
$L^2$-norm on the entire space. In the case where the space $\mathbb{R}^d$ is replaced by a finite cube $\Lambda_L$ we derive similar estimates.
A particular feature of our bound is that they are uniform as long as the coefficients of the operator are chosen from an appropriate ensemble,
they are quantitative and explicit with respect to the radius $\delta$, they are $L$-independent and stable under small shifts of the $\delta$-balls. Our proof applies to second order terms which have slowly varying coefficients on the relevant length scale.
The results can be also interpreted as special cases of \emph{uncertainty relations},
\emph{observability estimates}, or \emph{spectral inequalities}.
\end{abstract}
%
%
%
%
%
%
%
\section{Introduction}
\emph{How much can the amplitude of a function oscillate over a domain? How unevenly can the mass of the function be concentrated on different regions in space?}
This is -- phrased very sketchily -- the question we study in this paper. The functions considered are solutions of elliptic partial differential inequalities in various domains.
We are interested in a-priori estimates, 
which do not depend on the individual function considered, and are uniform with respect to (certain) variations of the domain and the variable coefficients of the elliptic operator. We measure the oscillations of the amplitude in terms of local $L^2$-norms.
There are three types of results:
\begin{enumerate}[(a)]
 \item a \emph{quantitative unique continuation principle} (or \emph{vanishing order estimate}) for solutions of variable coefficient elliptic partial differential equations, or inequalities,
 \item a \emph{sampling theorem} for solutions of variable coefficient elliptic partial differential inequalities on the entire Euclidean space $\mathbb{R}^d$, and
 \item an \emph{equidistribution theorem} for solutions of variable coefficient elliptic partial differential inequalities on cubes in $\mathbb{R}^d $ of odd integer length $L$.
\end{enumerate}
The first result concerns the vanishing order of a function and is in this sense local. However, as it is well known, global restrictions on the class of functions considered have a strong influence on the order of vanishing, cf.~\cite{DonnellyF-88, TaeuferTV-16}.
\par
The last two results (b) \& (c) are closely related. The sampling theorem can be understood as a version of the equidistribution theorem in the case where the side length of the cube is infinite and thus it equals the whole of $\mathbb{R}^d$.
Moreover, we derive similar estimates for linear combinations of (generalized) eigenfunctions of an elliptic partial differential operator, as long as the corresponding eigenvalues are sufficiently close together. Such functions obviously do not need to be solutions of a partial differential equation.
\par
We were led to derive such estimates motivated by our  previous studies of periodic or disordered physical systems modeled by partial differential operators.
The  coefficients of these operators are either periodic or stochastically homogeneous.  In \cite{Rojas-MolinaV-13} and \cite{Klein-13}, see also \cite{NakicTTV-15,NakicTTV-16-arxiv} for more general statements, such results have been derived for Schr\"odinger operators with periodic, quasiperiodic or random potentials. In this context they are a versatile tool for spectral analysis and can be exploited to establish
\emph{Anderson localization} for certain random models where this was not possible before. However, bounds of the type
(a), (b), and (c) above appear also in other contexts. For instance in \emph{quantum ergodicity} one is interested in delocalization and equidistribution properties of eigenfunctions \cite{BrooksML-15,AnantharamanM-15,Zelditch-92,BoechererSS-03},
in \emph{control theory} observability estimates and spectral inequalities play an important role, e.g.\ to estimate the control cost \cite{LebeauR-95,RousseauL-12},
on manifolds one studies the \emph{vanishing order of eigenfunctions} of the Laplace-Beltrami or a Schr\"odinger operator, cf.~\cite{DonnellyF-88,JerisonL-99,Kukavica-98,Bakri-13}. Finally, since our
theorems can be viewed as scale-uniform quantitative uncertainty principles for certain low dimensional subspaces, there is a relation to uniform uncertainty principles in
\emph{compressive sensing} as well.
We have no space to elucidate and dwell on these relations here, but refer to the survey paper \cite{TaeuferTV-16} for a detailed discussion.
\par
For a result which is already established for the Laplace operator one might wonder whether there is a straightforward extension to variable coefficient elliptic partial differential operators.
Indeed, for the questions at hand, if only the zero order term (interpreted as the potential) contains variable coefficients one can accommodate even local singularities, as demonstrated in  \cite{KleinT-16}. However, if the second order part has variable coefficients, the situation is different. While we can use the proof strategies of \cite{Rojas-MolinaV-13} and \cite{Klein-13}, the key tool, namely a Carleman estimate which holds for the Laplacian (or a Schr\"odinger operator) does not hold verbatim for variable coefficient operators. If one is striving for an optimal type of a Carleman estimate in the sense of \cite{EscauriazaV-03} or \cite{BourgainK-05} one cannot use simply the Carleman weight function of the Laplacian for other elliptic partial differential operators.
Rather, depending on Lipschitz and ellipticity constants of the
variable coefficients, one has to choose an adapted weight function. This has been observed in \cite{EscauriazaV-03} and quantitatively implemented in \cite{NakicRT-15-arxiv}. The latter refinement turns out to be crucial for the application in this note.
This leads to the condition in our theorems, that the coefficients are only allowed to vary slowly on the length scale which is determined by the equidistributed set. One way to satisfy this condition is to choose a dense sampling rate, another one to choose the Lipschitz constant of the coefficients in the partial differential equation sufficiently small.
It seems that with the existing Carleman estimates it is not possible to derive better results, cf.\, Remark
\ref{r:limitations} for more details.
\par
The rest of the paper is organized in the following way:
To illustrate our new results we resort to a comparison, and first consider the more transparent,
but simpler case of Schr\"odinger equations.
In the following section we state our new results, which are divided in two groups according to items (b), and (c) above.
In \S \ref{ss:qUC}
we spell out the main tool of our proof, namely a quantitative unique principle.
It corresponds to item (a) in the list above.
The remainder of Section 3 contains the proofs of the theorems in Section \ref{s:results},
whereas some technical aspects are deferred to an appendix.
On the technical level an interesting part of the paper will be the Remarks
\ref{rem:constant} and \ref{r:limitations}, where we discuss the innovations and limitations of our theorems and our approach.
\subsection*{Benchmark: Schr\"odinger operators}
The new results in the present paper are best understood when compared to what was recently established for the special case of (stationary) Schr\"odinger equations.
In this case only the zero order term is variable and the results are simpler to formulate, which we do next.
\par
For $L > 0$ we denote by $\Lambda_L = (-L/2 , L/2)^d \subset \mathbb{R}^d$ the cube with side length $L$, and by $\Delta_L$ the Laplace operator on $L^2 (\Lambda_L)$ subject to either Dirichlet, Neumann, or periodic boundary conditions.
Moreover, for a measurable and bounded $V : \mathbb{R}^d \to \mathbb{R}$ we denote by $V_L : \Lambda_L \to \mathbb{R}$ its restriction to $\Lambda_L$ given by $V_L (x)  = V (x)$ for $x \in \Lambda_L$, and by
\[
h_L = -\Delta_L + V_L \quad \text{on} \quad L^2 (\Lambda_L)
\]
the corresponding Schr\"odinger operator. For $\Omega \subset \mathbb{R}^d$ open and $\psi \in L^2 (\Omega)$ we denote by $\lVert \psi \rVert = \lVert \psi \rVert_\Omega$ the usual $L^2$-norm of $\psi$. If $\Gamma \subset \Omega$ we use the notation $\lVert \psi \rVert_\Gamma = \lVert \chi_\Gamma \psi \rVert_\Omega$. Moreover, we denote by $B (\rho) \subset \mathbb{R}^d$ the open ball in $\mathbb{R}^d$ with radius $\rho>0$ and center zero, by $B(x,\rho) \subset \mathbb{R}^d$ the open ball in $\mathbb{R}^d$ with radius $\rho>0$ and center $x \in \mathbb{R}^d$.
\begin{definition}
 Let $G > 0$ and $\delta > 0$. We say that a sequence $z_j \in \mathbb{R}^d$, $j \in (G \ZZ)^d$, is \emph{$(G,\delta)$-equidistributed}, if
 \[
  \forall j \in (G \ZZ)^d \colon \quad  B(z_j , \delta) \subset \Lambda_G + j .
\]
Corresponding to a $(G,\delta)$-equidistributed sequence $z_j \in \mathbb{R}^d$, $j \in (G \ZZ)^d$, we define for $L > 0$ the sets
\[
\equi_{\delta} = \bigcup_{j \in (G \ZZ)^d } B(z_j , \delta) \subset \mathbb{R}^d \quad \text{and} \quad
\equi_{\delta , L} = \bigcup_{j \in (G \ZZ)^d } B(z_j , \delta) \cap \Lambda_L \subset \Lambda_L .
\]
see Fig.~\ref{fig:equidistributed} for an illustration.
Note that the sets $S_\delta$, $S_{\delta,L}$  depend on $G$ and the choice of the $(G,\delta)$-equidistributed sequence.
\end{definition}
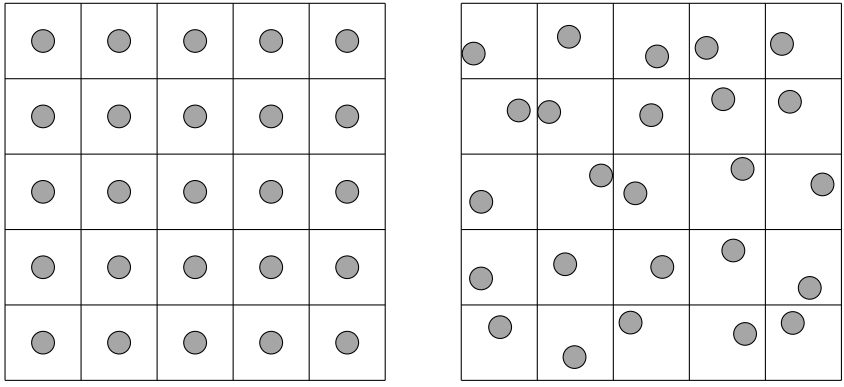
\begin{figure}[ht]\centering
\begin{tikzpicture}
\pgfmathsetseed{{\number\pdfrandomseed}}
\foreach \x in {0.5,1.5,...,4.5}{
  \foreach \y in {0.5,1.5,...,4.5}{
    \filldraw[fill=gray!70] (\x+rand*0.35,\y+rand*0.35) circle (0.15cm);
  }
}
\foreach \y in {0,1,2,3,4,5}{
\draw (\y,0) --(\y,5);
\draw (0,\y) --(5,\y);
}

\begin{scope}[xshift=-6cm]
\foreach \x in {0.5,1.5,...,4.5}{
  \foreach \y in {0.5,1.5,...,4.5}{
    \filldraw[fill=gray!70] (\x,\y) circle (1.5mm);
  }
}
\foreach \y in {0,1,2,3,4,5}{
  \draw (\y,0) --(\y,5);
  \draw (0,\y) --(5,\y);
}
\end{scope}
\end{tikzpicture}
\caption{\label{fig:equidistributed}%
Illustration of $S_{\delta,5}\subset  \Lambda_5 \subset \mathbb{R}^2$ for periodically (left)
and non-periodically (right) arranged $\delta$-equidistributed sequences.}
\end{figure}
\begin{theorem}[\cite{NakicTTV-15,NakicTTV-16-arxiv}] \label{thm:NTTV}
There is a constant $N = N(d)$, such that for all $G > 0$, all $\delta \in (0,G/2)$, all $(G,\delta)$-equidistributed sequences,
all measurable and bounded $V\colon {\mathbb{R}^d}\to \mathbb{R}$, all $L \in G\NN$, all $E_0 \geqslant 0$ and all $\phi \in \Ran (\chi_{(-\infty,E_0]}(h_{L}))$ we have
 \begin{equation*}
\lVert \phi \rVert_{\equi_{\delta , L}}^2
\geqslant C_{\sfUC}^{G} \lVert \phi \rVert_{\Lambda_L}^2 ,
\end{equation*}
where
\begin{equation*}
C_{\sfUC}^{G} = C_{\sfUC}^{G} (d, \delta ,  b  , \lVert V \rVert_\infty )
:=  \left(\frac{\delta}{G} \right)^{N \bigl(1 + G^{4/3} \lVert V \rVert_\infty^{2/3} + G \sqrt{ E_0} \bigr)} .
\end{equation*}
\end{theorem}
This extends previous results of \cite{CombesHK-07}, \cite{Rojas-MolinaV-13}, and \cite{Klein-13}. We denote by $\Delta : W^{2,2} (\mathbb{R}^d) \allowbreak \to L^2 (\mathbb{R}^d)$ the Laplace operator on $\mathbb{R}^d$.
\begin{theorem}[\cite{TautenhahnV-15}] \label{thm:TV}
There is a constant $N = N (d)$,
such that for all $E_0,G > 0$, all $\delta \in (0,G/2)$, all $(G,\delta)$-equidistributed sequences,
all measurable and bounded $V:\mathbb{R}^d \to \mathbb{R}$, and all intervals $I \subset (-\infty , E_0]$ with
\[
 \vert I \rvert \leqslant 2 \gamma \quad \text{where} \quad
 \gamma^2 = \frac{1}{2G^4} \left(\frac{\delta}{G}\right)^{N \bigl(1+ G^{4/3}(2\lVert V\rVert_\infty + E_0)^{2/3} \bigr)} ,
\]
and all $\psi \in \Ran (\chi_{I}(-\Delta + V))$ we have
\[
  \lVert \phi \rVert_{\equi_{\delta}}^2
\geqslant G^4 \gamma^2 \lVert \phi \rVert_{\mathbb{R}^d}^2 .
 \]
\end{theorem}
This is an adaptation of the main theorem of \cite{Klein-13} to the space $\mathbb{R}^d$.
Theorem~\ref{thm:TV} covers only short energy intervals $I$. An extension of this result to the case of arbitrary
compact intervals $I$ is not immediate by looking at the proof of \cite{NakicTTV-15,NakicTTV-16-arxiv}, but using a generalized
eigenfunction expansion it is likely that such an extension can be proven, which leads to the following
\begin{conjecture}
There is a constant $N = N (d)$,
such that for all $E_0,G > 0$, all $\delta \in (0,G/2)$, all $(G,\delta)$-equidistributed sequences,
all measurable and bounded $V:\mathbb{R}^d \to \mathbb{R}$, all $E_0>0$,
and all $\phi \in \Ran (\chi_{(-\infty,E_0]}(-\Delta + V))$ we have
 \begin{equation*}
\lVert \phi \rVert_{\equi_{\delta}}^2
\geqslant C_{\sfUC}^{G} \lVert \phi \rVert_{\mathbb{R}^d}^2
\end{equation*}
with $C_{\sfUC}^{G} $ as above.
\end{conjecture}
\begin{remark}
	Meanwhile, during the refereeing process, in a joint project with Ivica Naki\'c and Matthias T\"aufer, we have developed a proof of the above conjecture which will be published in \cite{NakicTTV-17-prep}.
\end{remark}
\section{Main results}\label{s:results}
Now we turn to the class of models which is treated in the theorems and proofs of the present paper.
Let $d \in \NN$ and $\Op$ be the second order partial differential 
expression
\begin{equation*} 
\Op u := - \diver (A \nabla u) + b^\T \nabla u + c u = -\sum_{i,j=1}^d \partial_i \left( a^{ij} \partial_j u \right) + \sum_{i=1}^d b_i \partial_i u + c u ,
\end{equation*}
where  $A : \mathbb{R}^d \to \mathbb{R}^{d \times d}$ with $A = (a^{ij})_{i,j=1}^d$, $b : \mathbb{R}^d \to \mathbb{C}^d$, $c : \mathbb{R}^d \to \mathbb{C}$, and $\partial_i$ denotes the $i$-th weak derivative.
We assume that $a^{ij} \equiv a^{ji}$ for all $i,j \in \{1,\ldots , d\}$, and that there are constants $\vartheta_1 \geqslant 1$ and $\vartheta_2 \geqslant 0$ such that for all $x,y \in \mathbb{R}^d$ and all $\xi \in \mathbb{R}^d$ we have
\begin{equation} \label{eq:elliptic}
\vartheta_1^{-1} \lvert \xi \rvert^2 \leqslant \xi^\T A (x) \xi \leqslant \vartheta_1 \lvert \xi \rvert^2
\quad\text{and}\quad
\lVert A(x) - A(y) \rVert_\infty \leqslant \vartheta_2 \lvert x-y \rvert .
\end{equation}
Moreover, we assume that $b,c \in L^\infty (\mathbb{R}^d)$.
Here we denote by $\lvert z \rvert$ the Euclidean norm of $z \in \mathbb{C}^d$, and by $\lVert M \rVert_\infty$ the row sum norm of a matrix $M \in \mathbb{C}^{d\times d}$.
%
%
%
\subsection{Sampling theorems on $\mathbb{R}^d$}
\label{sec:samplingR}
\begin{theorem}[Sampling Theorem]\label{thm:sampling}
Assume
\begin{equation} \label{ass:samplingG=1}
 \epsilon_1 := 1 - 33 \mathrm{e} d (\sqrt{d} + 2) \vartheta_1^{6} \vartheta_2   > 0 .
\end{equation}
Then for all measurable and bounded $V : \mathbb{R}^d \to \mathbb{R}$, all $\psi \in W^{2,2} (\mathbb{R}^d)$ and $\zeta \in L^2 (\mathbb{R}^d)$ satisfying $\lvert \Op \psi \rvert \leqslant \lvert V\psi \rvert + \lvert \zeta \rvert$ almost everywhere on $\mathbb{R}^d$, all $\delta \in (0,1/2)$ and all $(1,\delta)$-equi\-distri\-buted sequences we have
\[
 \lVert \psi \rVert_{\equi_\delta}^2 + \delta^2 \lVert \zeta \rVert_{\mathbb{R}^d}^2 \geqslant c_{\sfUC} \lVert \psi \rVert_{\mathbb{R}^d}^2 ,
\]
where
\begin{equation*} 
 c_{\sfUC} = d_1 \left( \frac{\delta}{d_2} \right)^{\frac{d_3}{\epsilon_1} \bigl( 1+  \lVert V \rVert_\infty^{2/3} + \lVert b \rVert_\infty^{2} + \lVert c \rVert_\infty^{2/3} \bigr) - \ln \epsilon_1}
\end{equation*}
with
\[
 d_1 = \frac{K_2 \vartheta_1^{-31/2-d} \mathrm{e}^{-10\vartheta_1}}{(1+\vartheta_2)^3},
 \quad
 d_2 = K_2 \vartheta_1^{2},
 \quad \text{and} \quad
 d_3 = K_2 \vartheta_1^{25} \mathrm{e}^{15 \vartheta_1} (1+\vartheta_2)^2 ,
\]
where $K_2$ is a positive constant depending only on the dimension.
\end{theorem}
By scaling, see Appendix~\ref{app:scaling}, we obtain the following variant of Theorem~\ref{thm:sampling} for $(G,\delta)$-equi\-distributed sequences.
\begin{theorem}\label{thm:samplingG}
Let $G > 0$ and assume
\begin{equation} \label{ass:sampling}
 \epsilon_2 := 1 - 33 \mathrm{e} d (\sqrt{d} + 2) \vartheta_1^{6} G \vartheta_2   > 0.
\end{equation}
Then for all measurable and bounded $V : \mathbb{R}^d \to \mathbb{R}$, all $\psi \in W^{2,2} (\mathbb{R}^d)$ and $\zeta \in L^2 (\mathbb{R}^d)$ satisfying $\lvert \Op \psi \rvert \leqslant \lvert V\psi \rvert + \lvert \zeta \rvert$ almost everywhere on $\mathbb{R}^d$, all $\delta \in (0,G/2)$ and all $(G,\delta)$-equi\-distri\-buted sequences we have
\[
 \lVert \psi \rVert_{\equi_\delta}^2 + \delta^2 G^2 \lVert \zeta \rVert_{\mathbb{R}^d}^2 \geqslant C_{\sfUC} \lVert \psi \rVert_{\mathbb{R}^d}^2 ,
\]
where
\begin{equation} \label{eq:definition-csfuc}
 C_{\sfUC} = D_1 \left( \frac{\delta}{G D_2} \right)^{\frac{D_3}{\epsilon_2} \bigl( 1 +  G^{4/3}\lVert V \rVert_\infty^{2/3} + G^2 \lVert b \rVert_\infty^{2} + G^{4/3}\lVert c \rVert_\infty^{2/3} \bigr) -\ln \epsilon_2}
\end{equation}
with
\[
 D_1 = \frac{K_2 \vartheta_1^{-31/2-d} \mathrm{e}^{-10\vartheta_1}}{(1+G\vartheta_2)^3},
 \quad
 D_2 = K_2 \vartheta_1^{2},
 \quad \text{and} \quad
 D_3 = K_2 \vartheta_1^{25} \mathrm{e}^{15 \vartheta_1} (1+G\vartheta_2)^2 ,
\]
where $K_2$ is a positive constant depending only on the dimension.
\end{theorem}
Obviously, $\epsilon_1$ in \eqref{ass:samplingG=1} equals $\epsilon_2$ with $G = 1$.
In the same way, $c_{\sfUC}$, $d_1$, $d_2$, and $d_3$ coincide with $C_{\sfUC}$, $D_1$, $D_2$, and $D_3$ in the case
$G = 1$.
\begin{remark}
Let us discuss condition \eqref{ass:samplingG=1} and \eqref{ass:sampling}, respectively.
Since the ellipticity constant $\vartheta_1$ is at least one,
it is required that the product $G \cdot \vartheta_2$ of the length scale and Lipschitz constant be small.
(If the scale is set to $G=1$ it means that $\vartheta_2$ should be small.
Compare also condition \eqref{ass:qUC} in Theorem~\ref{thm:qUC}.)
This means that the variation of the coefficients of the second order term should be sufficiently small
\emph{on the relevant scale} $G$.
Coinciding with physical intuition, it is possible to satisfy assumption \eqref{ass:samplingG=1}
by increasing the sampling rate appropriately, i.e.~by choosing the cubes $\Lambda_G +j$, which define the
$(G,\delta)$-equidistribution property, sufficiently small.

A condition like \eqref{ass:samplingG=1} or \eqref{ass:qUC}
naturally appears in the literature on unique continuation.
To satisfy it, authors usually assume that the radius/scale $R$, respectively $G$, is sufficiently small.
To prove the unique continuation property this is no restriction, since then only small balls are of interest.
\end{remark}

\begin{remark}[Scale free unique continuation constant $C_{\sfUC}$] \label{rem:constant}
To appreciate the theorem properly one wants to understand how
the constant $C_{\sfUC}$ depends on the model parameters.
First of all, we see that it is polynomial in the small radius $\delta$,
and that the exponent exhibited in \eqref{eq:definition-csfuc} is an estimate on the vanishing order,
e.g.\ as studied in \cite{DonnellyF-88,Kukavica-98}.
Second, one sees that the estimate is uniform as the potential $V$ or the coefficients $c$ and $b$ vary over ensembles
with uniformly bounded $L^\infty$-norm. No regularity properties, as encoded in Sobolev or total variation norms,
play a role here.
This is of importance, e.g., if one wants to derive results for random operators, where one is dealing not with one operator,
but a whole family of them, and aims at uniform bounds. See \cite{Rojas-MolinaV-13,Klein-13,BourgainK-13,NakicTTV-15,NakicTTV-16-arxiv}
for recent applications of this type.
Note that the dependence of the exponent in
$C_{\sfUC}$ is at worst quadratic with respect to $\|b\|_\infty$, $\|c\|_\infty$, $\|V\|_\infty$.
Thirdly, we see that Theorems~\ref{thm:L:per} to \ref{cor:L:Dir} hold for all odd $L\in \NN$ with a uniform constant $C_{\sfUC}$
independent of $L$. This is actually the reason why we call it \emph{scale-free unique continuation constant}.
Fourthly, note that the bound is stable under small shifts of the $\delta$-balls inside the periodicity cells.
All that needs to be satisfied is the geometric equidistribution property for the $\delta$-balls. This is crucial,
if one is considering a model which is ideally periodic, but one wants to make sure that results do not break
down if, more realistically,  small deviations from the ideal lattice structure are allowed.
Then there is the dependence on the parameter $\epsilon_2>0$, 
which measures the distance to the critical threshold value of zero.
This is not a natural model parameter, but a quantity which reflects the limitations of our Carleman estimate approach.
Finally, there are the model parameters $\vartheta_1$ and $\vartheta_2$.
One sees that while the dependence of the exponent in $C_{\sfUC}$ on $\vartheta_2$ is still of polynomial nature,
 the parameter $\vartheta_1$, i.e.~the ellipticity constant, enters in an exponential way. Both $\vartheta_1$ and $\vartheta_2$ influence $C_\sfUC$ in a monotone decreasing manner.
This is consistent with the results of Kukavica \cite[Theorem~5.1]{Kukavica-98}
where the vanishing rate  depends in a quadratic way on the sup-norm of the potential and in an exponential way on the coefficient functions of the second order part.
\end{remark}
A particular case where the assumption $\lvert \Op \psi \rvert \leqslant \lvert V\psi \rvert + \lvert \zeta \rvert$ in Theorem~\ref{thm:samplingG} is satisfied, is the case of an eigenfunction $\psi$. More generally, we formulate a corollary of Theorem~\ref{thm:samplingG} for functions in the range of some spectral projector of a self-adjoint realization of the differential expression $\Op$. We introduce the following assumption on the coefficient functions $b$ and $c$.
\begin{description}[\setleftmargin{0pt}\setlabelstyle{\bfseries}]
 \item[(SA):] We have $b = \mathrm{i} \tilde b$ and $c = \tilde c + \mathrm{i} \diver \tilde b / 2$ for some bounded $\tilde b , \tilde c \in L^\infty (\mathbb{R}^d)$.
\end{description}
We define the differential operator $H : W^{2,2} (\mathbb{R}^d) \to L^2 (\mathbb{R}^d)$, $H \psi = \Op \psi$. If assumption (SA) is satisfied, then $H$ is a self-adjoint operator in $L^2 (\mathbb{R}^d)$.
\begin{theorem}\label{cor:sampling}
Let $G > 0$ and assume (SA) and Ineq.~\eqref{ass:sampling}.
Then for all $E \in \mathbb{R}$, all $\delta \in (0,G/2)$, all $(G,\delta)$-equi\-distri\-buted sequences, and all $\psi \in \operatorname{Ran} \chi_{[E-\gamma , E + \gamma]} (H)$ with
\[
 \gamma^2 = \frac{D_1}{G^4}
  \left( \frac{\delta}{G D_2} \right)^{\frac{D_3}{\epsilon_2} \bigl( 1 +  G^{4/3} \lvert E \rvert^{2/3} + G^2 \lVert b \rVert_\infty^{2} + G^{4/3} \lVert c \rVert_\infty^{2/3} \bigr)-\ln \epsilon_2}
\]
we have
\[
 \lVert \psi \rVert_{\equi_\delta}^2  \geqslant \frac{C_{\sfUC}}{2} \lVert \psi \rVert_{\mathbb{R}^d}^2
 \quad\text{with}\quad
 C_\sfUC = D_1
  \left( \frac{\delta}{G D_2} \right)^{\frac{D_3}{\epsilon_2} \bigl( 1 +  G^{4/3} \lvert E \rvert^{2/3} + G^2 \lVert b \rVert_\infty^{2} + G^{4/3} \lVert c \rVert_\infty^{2/3} \bigr) -\ln \epsilon_2} ,
\]
where $D_1$, $D_2$ and $D_3$ are given in Theorem~\ref{thm:samplingG}.
\end{theorem}
\begin{remark}[Relation between the condition on the parameters $\vartheta_1,\vartheta_2$ and the constants $ C_\sfUC $]
\label{r:limitations}
It is natural to wonder whether the restriction
\eqref{ass:sampling},
respectively
\eqref{ass:qUC}, is necessary to derive a quantitative unique continuation estimate
as it is expressed in Theorem~\ref{thm:qUC} and Lemma~\ref{lemma:constants}, i.e.\ where we have a powerlike vanishing behavior and
where the bound on vanishing order is a polynomial function of
$\|b\|_\infty$, $\|c\|_\infty$, $\|V\|_\infty$, it seems that we have to use a Carleman weight function as in
\cite{NakicRT-15-arxiv}, see Theorem~\ref{thm:carleman}. In that case the restriction on the Lipschitz and ellipticity constant enters naturally.
\par
Of course it would be possible to use weight functions which are not of polynomial type.
We believe the that in this case it would be possible to remove the conditions on the slow variation of the coefficients \eqref{ass:qUC}, respectively \eqref{ass:sampling}.
However, it is unclear whether we could obtain with this approach quantities $C_\qUC$ in Lemma~\ref{lemma:constants} and $C_\sfUC$ in the subsequent theorems where in the exponent appear only linear and quadratic expressions of $\lVert V \rVert_\infty$, $\lVert b \rVert_\infty$, and $\lVert c \rVert_\infty$. We plan to investigate this approach in a sequel paper.
Finally, one could try to find approaches different from Carleman estimates to obtain estimates on the vanishing order.
In dimension one and two there should be enough alternative tools to achieve this goal.
\end{remark}
%

\subsection{Equidistribution theorems on $\Lambda_L$}
\label{sec:samplingL}
In this section we will consider functions $\psi \in W^{2,2} (\Lambda_L)$ and $\zeta \in L^2 (\Lambda_L)$, $L > 0$, satisfying $\lvert \Op \psi \rvert \leqslant \lvert V \psi \rvert + \lvert \zeta \rvert$ almost everywhere on $\Lambda_L$.
In order to define appropriate extensions of such functions we will assume that $\psi$ satisfies Dirichlet or periodic boundary conditions on the sides of $\Lambda_L$,
and denote by $\Dom (\Delta_L^{\mathrm{Dir}})$ and $\Dom (\Delta_L^{\mathrm{per}})$ the domain of the Laplace operator on $\Lambda_L$
subject to Dirichlet or periodic boundary conditions, respectively. Moreover, depending on the boundary conditions, we introduce for $L > 0$ the following assumptions on the coefficients $a^{ij}$, $i,j\in \{ 1,\ldots , d \}$.
\begin{description}[\setleftmargin{0pt}\setlabelstyle{\bfseries}]
 \item[(Dir)]\label{Dir}
For all $i,j \in \{1,\ldots,d\}$ with $i \not = j$ the coefficient function $a^{ij}$ vanishes on the sides of $\Lambda_L$.
 \item[(Per)]\label{Per} For all $i,j \in \{1,\ldots,d\}$ the coefficient function $a^{ij}$ satisfies periodic boundary conditions on the sides of $\Lambda_L$.
\end{description}
\begin{theorem}[Equidistribution Theorem]\label{thm:L:per}
Let $L,G > 0$ and assume \eqref{ass:sampling}, $L/G \in \NN$ is odd, and (Per).
Then for all measurable and bounded $V : \Lambda_L \to \mathbb{R}$, all  $\psi\in\Dom(\Delta_L^{\mathrm{per}})$ and $\zeta \in L^2 (\Lambda_L)$ satisfying $\lvert \Op \psi \rvert \leqslant \lvert V\psi \rvert + \lvert \zeta \rvert$ almost everywhere on $\Lambda_L$, all $\delta \in (0,G/2)$ and all $(G,\delta)$-equi\-distri\-buted sequences we have
\begin{equation*} 
\lVert \psi \rVert_{\equi_{\delta,L}}^2 + \delta^2 G^2 \lVert \zeta \rVert_{\Lambda_L}^2
\geqslant C_{\sfUC} \lVert \psi \rVert_{\Lambda_L}^2 ,
\end{equation*}
where $C_{\sfUC} = C_{\sfUC} (d ,G, \delta , \vartheta_1 , \vartheta_2 , \lVert V \rVert_\infty , \lVert b \rVert_\infty , \lVert c \rVert_\infty)$ is given in Theorem~\ref{thm:samplingG}.
\end{theorem}
\begin{theorem}[Equidistribution Theorem bis]\label{thm:L:Dir}
Let $L,G > 0$ and assume \eqref{ass:sampling}, $L/G \in \NN$ is odd, and (Dir).
Then for all measurable and bounded $V : \Lambda_L \to \mathbb{R}$, all  $\psi\in\Dom(\Delta_L^{\mathrm{Dir}})$ and $\zeta \in L^2 (\Lambda_L)$ satisfying $\lvert \Op \psi \rvert \leqslant \lvert V\psi \rvert + \lvert \zeta \rvert$ almost everywhere on $\Lambda_L$, all $\delta \in (0,G/2)$ and all $(G,\delta)$-equi\-distri\-buted sequences we have
\begin{equation*} 
\lVert \psi \rVert_{\equi_{\delta,L}}^2 + \delta^2 G^2 \lVert \zeta \rVert_{\Lambda_L}^2
\geqslant C_{\sfUC} \lVert \psi \rVert_{\Lambda_L}^2 ,
\end{equation*}
where $C_{\sfUC} = C_{\sfUC} (d , G , \delta , \vartheta_1 , \vartheta_2 , \lVert V \rVert_\infty , \lVert b \rVert_\infty , \lVert c \rVert_\infty)$ is given in Theorem~\ref{thm:samplingG}.
\end{theorem}
We define for $L > 0$ the differential operators $H_L^{\mathrm per} : \Dom (\Delta_L^{\mathrm{per}}) \to L^2 (\Lambda_L)$ and $H_L^{\mathrm Dir} : \Dom (\Delta_L^{\mathrm{Dir}}) \to L^2 (\Lambda_L)$
by $H_L^{\mathrm per} \psi = \Op \psi$ and $H_L^{\mathrm Dir} \psi = \Op \psi$. If assumption (SA) from Section~\ref{sec:samplingR} is satisfied,
then $H_L^{\mathrm per}$ and $H_L^{\mathrm Dir}$ are self-adjoint operators in $L^2 (\Lambda_L)$.
Here are two analogs of Theorem~\ref{cor:sampling} for operators defined on boxes $\Lambda_L$.
\begin{theorem}\label{cor:L:per}
Let $L,G > 0$ and assume \eqref{ass:sampling}, (SA), $L/G  \in \NN$ is odd, and (Per).
Then for all $E \in \mathbb{R}$, all $\delta \in (0,G/2)$, all $(G,\delta)$-equi\-distri\-buted sequences, and all $\psi \in \operatorname{Ran} \chi_{[E-\gamma , E + \gamma]} (H_L^{\mathrm{per}})$ with
\[
 \gamma^2 =  \frac{D_1}{G^4} \left( \frac{\delta}{G D_2} \right)^{\frac{D_3}{\epsilon_2} \bigl( 1 + G^{4/3} \lvert E \rvert^{2/3} + G^2 \lVert b \rVert_\infty^{2} + G^{4/3} \lVert c \rVert_\infty^{2/3} \bigr) -\ln \epsilon_2}
\]
we have
\[
 \lVert \psi \rVert_{\equi_{\delta,L}}^2  \geqslant \frac{C_\sfUC}{2} \lVert \psi \rVert_{\Lambda_L}^2
 \quad\text{with}\quad
 C_\sfUC = D_1 \left( \frac{\delta}{G D_2} \right)^{\frac{D_3}{\epsilon_2} \bigl( 1 + G^{4/3} \lvert E \rvert^{2/3} + G^2 \lVert b \rVert_\infty^{2} + G^{4/3} \lVert c \rVert_\infty^{2/3} \bigr) -\ln \epsilon_2}  ,
\]
where $D_1$, $D_2$ and $D_3$ are given in Theorem~\ref{thm:samplingG}.
\end{theorem}
\begin{theorem}\label{cor:L:Dir}
Let $L,G > 0$ and assume \eqref{ass:sampling}, (SA), $L/G \in \NN$ is odd, and (Dir).
Then for all $E \in \mathbb{R}$, all $\delta \in (0,G/2)$, all $(G,\delta)$-equi\-distri\-buted sequences, and all $\psi \in \operatorname{Ran} \chi_{[E-\gamma , E + \gamma]} (H_L^{\mathrm{Dir}})$ with
\[
 \gamma^2 =  \frac{D_1}{G^4} \left( \frac{\delta}{G D_2} \right)^{\frac{D_3}{\epsilon_2} \bigl( 1 + G^{4/3} \lvert E \rvert^{2/3} + G^2 \lVert b \rVert_\infty^{2} + G^{4/3} \lVert c \rVert_\infty^{2/3} \bigr) -\ln \epsilon_2}
\]
we have
\[
 \lVert \psi \rVert_{\equi_{\delta,L}}^2  \geqslant \frac{C_\sfUC}{2} \lVert \psi \rVert_{\Lambda_L}^2
 \quad\text{with}\quad
 C_\sfUC = D_1 \left( \frac{\delta}{G D_2} \right)^{\frac{D_3}{\epsilon_2} \bigl( 1 + G^{4/3} \lvert E \rvert^{2/3} + G^2 \lVert b \rVert_\infty^{2} + G^{4/3} \lVert c \rVert_\infty^{2/3} \bigr) -\ln \epsilon_2}  ,
\]
where $D_1$, $D_2$ and $D_3$ are given in Theorem~\ref{thm:samplingG}.
\end{theorem}
\begin{remark}[Additional boundary conditions for the coefficients of $A$]
In the equidistribution theorem there appear special boundary conditions on the coefficient matrix $ A \colon \mathbb{R}^d \to \mathbb{R}^{d \times d}$  which do not feature in the sampling theorem.
The reason is, that in the proof we need to extend the differential equation and its solution
to a larger set, in order to avoid dealing with regions near the boundary of the domain. It is a general principle
that in such regions estimates on gradients and other derivatives are harder to obtain than in the interior.
(In fact, we extend the equation to the whole of $\mathbb{R}^d$ for simplicity.)
Now the extension from $\Lambda_L$ to $\mathbb{R}^d$ is done with a simple mirroring procedure,
where we have to match all of the relevant interior and exterior derivatives.
For a Laplace operator this is always possible. In the case of an elliptic differential operator
with mixed derivatives present, this leads, in general, to an overdetermined system of linear equations.
The conditions (Dir) and (Per) we impose eliminate some of the equations, so that under these conditions a solution
to system of linear equations always exists. A more flexible way to extend the solutions outside of the domain $\Lambda_L$
could possibly allow the boundary  restrictions (Dir), respectively (Per), to be lifted.
However, we do not regard this as the primary challenge to improve our results in this paper,
but subordinate to the question, whether one can allow coefficients with arbitrary large
Lipschitz constants.
\end{remark}
%
%
%
\section{Proofs}
%
%
%
%
\subsection{Quantitative unique continuation}\label{ss:qUC}
The first step in our proofs is the following quantitative unique continuation principle which we formulate next.
\begin{theorem}[Quantitative unique continuation theorem] \label{thm:qUC}
Let $D_0,R \in (0,\infty)$, $\delta \in (0, 2 R)$ $K_V, \beta \in [0,\infty)$, and assume
\begin{equation} \label{ass:qUC}
 \epsilon_0 := 1 - 33 \mathrm{e} d R \vartheta_1^{6} \vartheta_2  > 0.
\end{equation}
Then there is a constant $C_{\qUC} = C_{\qUC} (d , \vartheta_1 , \vartheta_2 , R , D_0 , K_V , \lVert b \rVert_\infty, \lVert c \rVert_\infty , \delta , \beta) > 0$, such that for any $\Omega\subset \mathbb{R}^d$ open, $x \in \Omega$ and $\Theta \subset \Omega$ measurable and satisfying
 \[
 \Theta \subset \overline{B (x,R)} \setminus B (x,\delta / 2) \quad
 \text{and} \quad
 B(x,2 \mathrm{e} \vartheta_1 R+2D_0) \subset \Omega,
\]
any measurable $V : \Omega \to [-K_V , K_V]$, any $\zeta \in L^2 (\Omega)$ and $\psi \in W^{2,2} (\Omega)$ satisfying the differential inequality
\begin{equation*} 
 \lvert \Op \psi \rvert \leqslant \lvert V\psi \rvert + \lvert \zeta \rvert \quad \text{a.e.\ on $\Omega$} \quad \text{as well as} \quad \frac{\lVert \psi \rVert_\Omega^2}{\lVert \psi \rVert_\Theta^2} \leqslant \beta ,
\end{equation*}
we have
\[
 \lVert \psi \rVert_{B(x,\delta)}^2 + \delta^2 \lVert \zeta \rVert^2_\Omega \geqslant C_{\qUC} \lVert \psi \rVert_{\Theta}^2  .
\]
\end{theorem}
Note that $\Theta$ and $B (x,\delta)$ may overlap. The constant $C_\qUC$ is given explicitly in Eq.~\eqref{eq:constant_qUC}.

\begin{remark}
Inspired by \cite{BourgainK-05} several related quantitative unique continuation principles
have been proven in the literature on (random) Schr\"odinger operators, see
\cite{GerminetK-13,BourgainK-13,Rojas-MolinaV-13,TaeuferT-17}.
The application of these estimates are manifold, including a Wegner estimate for alloy type Schr\"odinger operators with small support, the log H\"older
continuity of the integrated density of states for general Schr\"odinger operators, and localization on various energy/disorder regimes.
Our result is an extension of these results to variable
coefficient divergence type operators.

The inequalities are loosely related to so called three circle annuli inequalities as they are often used in the literature on control theory
and harmonic analysis on compact manifolds, see for instance \cite{RousseauL-12,Bakri-13}.
\end{remark}

In Appendix~\ref{sec:CqUC} we give an estimate on $C_{\qUC}$ under several assumptions on the parameters. This is formulated in the following lemma.
\begin{lemma} \label{lemma:constants}
 Let $2D_0 = R \geqslant 1$, $\delta < 2$ and $\epsilon_0 > 0$. Then,
 \[
  C_{\qUC} \geqslant C_1 \left( \frac{\delta}{C_2 R} \right)^{\frac{C_3}{\epsilon_0} \left(1 + \lVert V \rVert_\infty^{2/3} + \lVert b \rVert_\infty^{2} + \lVert c \rVert_\infty^{2/3} \right) R^3 -\ln \epsilon_0 + \ln \beta} ,
 \]
with
\[
 C_1 = \frac{K_1 \vartheta_1^{-31/2} \mathrm{e}^{-10\vartheta_1}}{(1+\vartheta_2)({\vartheta_1} + \vartheta_2^2)},
 \quad
 C_2 = 10\mathrm{e}\vartheta_1^{2},
 \quad \text{and} \quad
 C_3 = K_1  \vartheta_1^{25} \mathrm{e}^{15 \vartheta_1} (1+\vartheta_2)^2 ,
\]
where $K_1$ is a positive constant depending only on the dimension.
\end{lemma}
The explicit form of the constant $C_{\qUC} $ may appear complicated and, indeed,
merits a detailed discussion. Among others it gives an estimate on the vanishing order of $\psi$.
Since $C_{\qUC}$  shares its important features with the constant
$C_{\sfUC}$ appearing in the next theorem, we will not discuss  $C_{\qUC}$ separately but refer to Remark \ref{rem:constant}.

Now the proof of Theorems~\ref{thm:sampling} to \ref{cor:L:Dir} follow.
%
%
%
%
\subsection{Carleman estimate and Cacciopoli inequality}
We start with a formulation of the the Cacciopoli inequality,
which may be found in \cite{Rojas-MolinaV-13} in the case where $\Op = -\Delta$.
\begin{lemma}[Cacciopoli inequality] \label{lemma:Cacciopoli}
Let $\Omega \subset \mathbb{R}^d$ be open, $V : \Omega \to \mathbb{R}$ bounded and measurable, $\zeta \in L^2 (\Omega)$, $\psi \in W^{2,2} (\Omega)$ satisfying $\lvert \Op \psi \rvert \leqslant \lvert V \psi \rvert + \lvert \zeta \rvert$ almost everywhere on $\Omega$, $0 \leqslant r_1 < r_2$, $r \in (0,\infty)$, $S = B (r_2) \setminus \overline{B (r_1)}$, $S^+ = B (r_2+r) \setminus \overline{B (r_1 - r)}$ and assume $S^+ \subset \Omega$. Then there is an absolute constant $C' \geqslant 1$ such that
 \[
  \int_{S} {\nabla \psi^\T A \nabla \overline{\psi}} \leqslant \left(2 \lVert V \rVert_\infty^2 + 1 + 2 \vartheta_1 \lVert b \rVert_\infty^2 + \frac{8\vartheta_1 C'}{r^2} + 2 {\lVert c \rVert_\infty} \right) \int_{S^+} \lvert \psi \rvert^2  + 2 \int_{S^+} \lvert \zeta \rvert^2 .
 \]
For the prefactors appearing on the right hand side we use the symbol $\Cac_{r} = \Cac_r (V , b , c, \vartheta_1)$.
\end{lemma}
\begin{proof}
We use $0 \leqslant \lVert x-y \rVert^2 = \lVert x \rVert^2 + \lVert y \rVert^2 - 2 \RE \langle x,y \rangle$, Green's theorem and Cauchy Schwarz, and obtain for all real-valued $\eta \in C_{\mathrm{c}}^\infty (\Omega)$
\begin{align*}
  \lVert \eta \Op \psi \rVert^2 &+ \lVert \eta \psi \rVert^2
\geqslant 2 \RE \langle \Op \psi , \eta^2 \psi \rangle \\
&= 4 \RE \langle \eta \nabla \psi , A \psi \nabla \eta \rangle
+
2 \langle \nabla \psi , \eta^2 A \nabla \psi \rangle + 2\RE \langle b^\T \nabla \psi , \eta^2 \psi \rangle + 2 {\RE}\langle c \psi , \eta^2 \psi \rangle \\
  &\geqslant {2\langle \nabla \psi , \eta^2 A \nabla \psi \rangle} -4 \lVert \eta A^{1/2} \nabla \psi \rVert  \lVert \psi A^{1/2} \nabla \eta \rVert  - 2 \lVert \eta b^\T \nabla \psi \rVert \lVert \eta \psi \rVert + 2{\RE} \langle c \psi , \eta^2 \psi \rangle  .
\end{align*}
Since $2ab \leqslant s a^2 + s^{-1} b^2$, we have for all $s,t > 0$
\begin{multline*}
 \lVert \eta \Op \psi \rVert^2 + \lVert \eta \psi \rVert^2
  \\ \geqslant
  - \frac{2}{s}\lVert \psi A^{1/2} \nabla \eta \rVert^2 + (2-2s){\langle \nabla \psi , \eta^2 A \nabla \psi \rangle} - \frac{1}{t}\lVert \eta b^\T \nabla \psi \rVert^2 - t\lVert \eta \psi \rVert^2 + 2{\RE}\langle c \psi , \eta^2 \psi \rangle .
\end{multline*}
We choose $s = 1/4$, $t = 2 \vartheta_1 \lVert b \rVert_\infty^2$ and obtain by using $\lVert \eta b^\T \nabla \psi \rVert^2 \leqslant \lVert b \rVert_\infty^2 \lVert \eta \nabla \psi \rVert^2$
\begin{equation}\label{eq:gradient}
 {\langle \nabla \psi , \eta^2 A \nabla \psi \rangle}
 \leqslant \lVert \eta \Op \psi \rVert^2 +(1+2 \vartheta_1 \lVert b \rVert_\infty^2) \lVert \eta \psi \rVert^2 + 8\lVert \psi A^{1/2} \nabla \eta \rVert^2 + 2 \lVert c \rVert_\infty \langle  \psi , \eta^2 \psi \rangle .
\end{equation}
Now we choose a radially symmetric function $\eta \in C_{\mathrm c}^\infty (\Omega)$ with $\eta \in [0,1]$, $\supp \eta = S^+$, $\eta \equiv 1$ on $S$, and $\lvert \nabla \eta \rvert \leqslant \sqrt{C'} / r$
with some absolute constant $C' > 0$. Hence, by our assumption $\lvert \Op\psi \rvert \leqslant \lvert V\psi \rvert + \lvert \zeta \rvert$, ellipticity, and our Lipschitz condition, the statement of the lemma follows from Ineq.~\eqref{eq:gradient}.
\end{proof}
Now we cite a Carleman estimate from \cite{NakicRT-15-arxiv}. A particular feature of this Carleman estimate is that the weight function and the corresponding constants are given explicitly. For us it is of particular importance to track the dependence on the ellipticity and Lipschitz constants $\vartheta_1$ and $\vartheta_2$, respectively.
Carleman estimates of this type have been given before in \cite{EscauriazaV-03,KenigSU-11} and \cite{BourgainK-05}.
However, they did not derive the dependence on the weight function and the  constants in the Carleman estimate on $\vartheta_1$ and $\vartheta_2$. Or they concerned only the pure Laplacian in the first place.
\par
For $\mu,\rho > 0$ we introduce a function $w_{\rho , \mu} : \mathbb{R}^d \to [0,\infty)$ by
\[
 w_{\rho , \mu} (x) := \varphi (\sigma (x / \rho)) ,
\]
where $\sigma:\mathbb{R}^d \to [0,\infty)$ and $\varphi : [0,\infty) \to [0,\infty)$ are given by
\begin{equation} \label{eq:weight}
\sigma(x):= \left(x^\T A(0)^{-1} x \right)^{1/2}, \quad \text{and} \quad
\varphi(r):= r \exp \left( - \int_0^r\frac{1 - e^{-\mu t}}{t} \mathrm{d} t \right).
\end{equation}
We recall the upper and lower bounds on $\varphi$ given in \cite{NakicRT-15-arxiv}. It is obvious that $\varphi (r) \le r$. 
For a lower bound we distinguish three cases. If $\sqrt{\vartheta_1} \mu \leqslant 1$ we have $\varphi (r) \geqslant r \mathrm{e}^{-\mu r}$. If $\sqrt{\vartheta_1} \mu \geqslant 1$ and $r \leqslant 1/\mu$ we use $1-\mathrm{e}^{\mu t} \leqslant \mu t$ and obtain $\varphi (r) \geqslant r \mathrm{e}^{-{\mu r}} \geqslant r (\mathrm{e} \sqrt{\vartheta_1} \mu)^{-1}$. If $\sqrt{\vartheta_1} \mu \geqslant 1$ and $r > 1/\mu$ we split the integral according to $[0,1/\mu]$ and $[1/\mu , r]$. On the first component we use $1-\mathrm{e}^{-\mu t} \leqslant \mu t$, on the second one $1-\mathrm{e}^{-\mu t} \leqslant 1$. This way we obtain $\varphi (r) \geqslant 1/(\mathrm{e} \mu)$ if $r > 1/\mu$. It follows that for all $r \in [0,\sqrt{\vartheta_1}]$
\[
 \varphi (r) \geqslant \frac{r}{\mu_1} ,
 \quad \text{where} \quad 
 \mu_1 = \begin{cases}
          \mathrm{e}^{\sqrt{\vartheta_1} \mu}& \text{if $\sqrt{\vartheta_1} \mu \leqslant 1$,}\\
          \mathrm{e} \sqrt{\vartheta_1} \mu& \text{if $\sqrt{\vartheta_1} \mu \geqslant 1$}.
         \end{cases} 
\]
Hence, the function $w_{\rho , \mu}$ satisfies
\begin{equation} \label{eq:weightbound2}
 \forall x \in B(\rho) \colon \quad
  \frac{\vartheta_1^{-1/2} \lvert x \rvert}{\rho \mu_1} \leqslant \frac{\sigma (x)}{\rho  \mu_1}  \leqslant
 w_{\rho , \mu}(x)
 \leqslant
 \frac{\sigma (x)}{\rho}
 \leqslant
 \frac{\sqrt{\vartheta_1} \lvert x \rvert}{\rho} ,
\end{equation}
where
\[
\mu_1 = \begin{cases}
          \mathrm{e}^{\sqrt{\vartheta_1} \mu}& \text{if $\sqrt{\vartheta_1} \mu \leqslant 1$,}\\
          \mathrm{e} \sqrt{\vartheta_1} \mu& \text{if $\sqrt{\vartheta_1} \mu > 1$}.
         \end{cases}
\]
\begin{theorem}[\cite{NakicRT-15-arxiv}] \label{thm:carleman}
Let $\rho > 0$ and $\mu > 33 d \vartheta_1^{11/2} \vartheta_2 \rho$.
Then there are constants $\alpha_0 = \alpha_0 (d, \rho  , \vartheta_1 , \vartheta_2, \mu , \lVert b \rVert_\infty , \lVert c \rVert_\infty) > 0$ and $C = C (d , \vartheta_1 , \rho \vartheta_2, \mu) > 0$, such that for all $\alpha \geqslant \alpha_0$ and all $u \in W^{2,2} (\mathbb{R}^d)$ with support in $B (\rho) \setminus \{0\}$ we have
\begin{equation*}
 \int_{\mathbb{R}^d} \left( \alpha \rho^2 w_{\rho , \mu}^{1-2\alpha} \nabla u^\T A \nabla u + \alpha^3 w_{\rho , \mu}^{-1-2\alpha} \lvert u \rvert^2  \right) \leqslant C \rho^4 \int_{\mathbb{R}^d} w_{\rho , \mu}^{2-2\alpha} \lvert \Op u \rvert^2 .
\end{equation*}
\end{theorem}
\begin{remark}
Note that $\varphi$ in \eqref{eq:weight} is a product of two terms: The first is linear and universal, the second contains the dependence on the operator via the parameter $\mu$.
If one could derive a Carleman estimate valid for all $A$ as in \eqref{eq:elliptic}
with a weight function independent of
$\vartheta_1$ and $\vartheta_2$ , we would be able to drop the smallness
assumptions \eqref{ass:samplingG=1}, \eqref{ass:sampling},  and \eqref{ass:qUC}.
However, we were neither able to prove
nor to find in the literature such a Carleman estimate.
\end{remark}

\begin{remark}
Upper bounds for the constants $C$ and $\alpha_0$ are known explicitly, see \cite{NakicRT-15-arxiv}. In the case where $b$ and $c$ are identically zero, the conclusion of Theorem~\ref{thm:carleman} holds with $C = \tilde C$ and $\alpha_0 = \tilde \alpha_0$ satisfying the upper bounds
\begin{align*}
 \tilde C &\leqslant 2d^2  \vartheta_1^8 \mathrm{e}^{4\mu\sqrt{\vartheta_1}} \mu_1^4 \left( 3 \mu^2 + (9\rho \vartheta_2 + 3)\mu + 1 \right) C_\mu^{-1} \\
 \intertext{and}
 \tilde \alpha_0 & \leqslant 11 d^4 \vartheta_1^{33/2} \mathrm{e}^{6\mu\sqrt{\vartheta_1}}   \mu_1^6  (3\rho\vartheta_2 + \mu + 1)^2 \left(1  + \mu (\mu + 1) C_\mu^{-1}   \right) ,
\end{align*}
where $C_\mu = \mu - 33 d \vartheta_1^{11/2} \vartheta_2 \rho$. In the general case where $b,c \in L^\infty (B(\rho))$ the conclusion of the theorem holds with
\[
 C = 6 \tilde C
 \quad \text{and} \quad
 \alpha_0 =\max \left\{ \tilde\alpha_0, C \rho^2  \lVert  b  \rVert_{\infty}^2 \vartheta_1^{3/2},
C^{1/3} \rho^{4/3} \lVert c \rVert_{\infty}^{2/3}\sqrt{\vartheta_1}   \right\} .
\]
\end{remark}
%
%
%
%
%
\subsection{Proof of Theorem~\ref{thm:qUC}}
\begin{proof}[Proof of Theorem~\ref{thm:qUC}]
We follow \cite[Proof of Theorem~3.1]{Rojas-MolinaV-13}.
 For convenience we assume $x = 0$, hence $\Omega \supset B (2 \mathrm{e} \vartheta_1 R + 2 D_0)$. The general case follows by translation. We choose a function $\eta : \mathbb{R}^d \to [0,1]$, $\eta \in C_{\mathrm c}^\infty (\Omega)$, depending only on $\lvert x \rvert$ satisfying
 \begin{align}
 \eta (x)
 &=
 \begin{cases}
  0 & \text{if $x \in \overline{B(\delta / 4)} \cup B(2 \mathrm{e} \vartheta_1 R + D_0)^{\mathrm c}$}, \\
  1 & \text{if $x \in \overline{B(2 \mathrm{e} \vartheta_1 R)} \setminus B(\delta / 2)$},
 \end{cases}    \nonumber \\
 \intertext{and}
 \max \left( \lVert \nabla \eta \rVert_\infty , \lVert \Delta \eta \rVert_\infty \right) &\leqslant
 \begin{cases}
  \left(M / \delta \right)^2 & \text{if $x \in \overline{B(\delta / 2)} \setminus B(\delta / 4)$}, \\
  \left(M / D_0\right)^2 & \text{if $x \in \overline{B(2 \mathrm{e} \vartheta_1 R+D_0)} \setminus B(2\mathrm{e} \vartheta_1 R )$}
 \end{cases} \label{eq:prop_eta} ,
\end{align}
where $M = M(d) \in (0,\infty)$ is a constant depending only on the dimension.
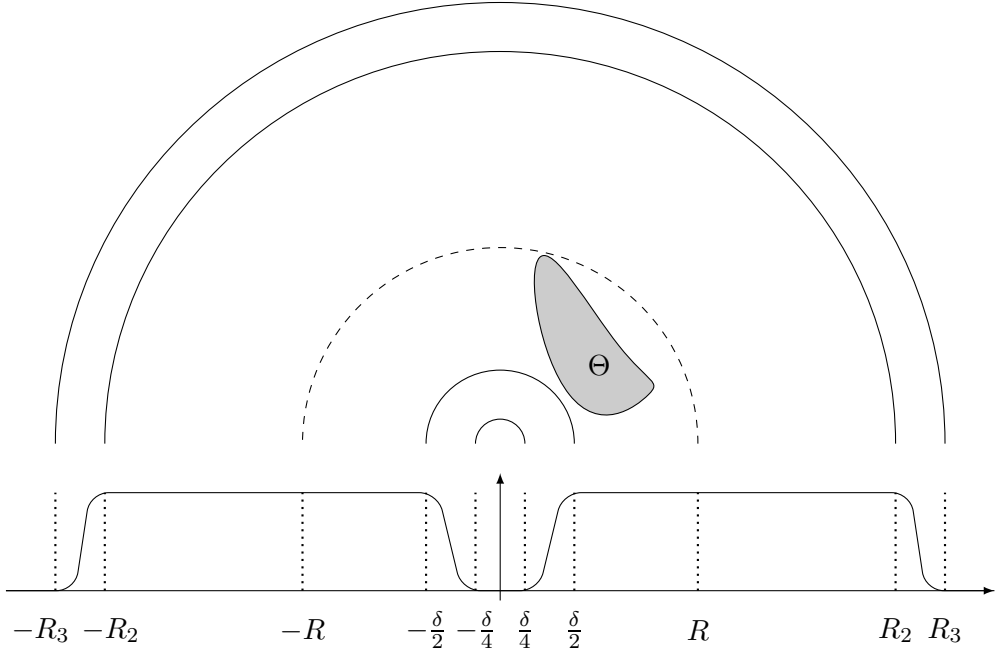
\begin{figure}[ht]\centering
\begin{tikzpicture}[yscale=1.3,xscale=1.3]
 \draw[-latex] (-5,0)--(5,0);
 \draw[-latex] (0,-0.1)--(0,1.2);
 \draw[rounded corners=7pt]
 (-5,0)--(-4.3,0)--(-4.15,1)--(-0.63,1)--(-0.39,0)--(0.39,0)--(0.63,1)--(4.15,1)--(4.3,0)--(5,0);
 \draw[dotted, thick] (-0.25,-0)--(-0.25,1);
 \draw[dotted, thick] (0.25,-0)--(0.25,1);
 \draw (0.25,-0.4) node {$\frac{\delta}{4}$};
 \draw (-0.25,-0.4) node {$-\frac{\delta}{4}$};
 \draw[dotted, thick] (-0.75,-0)--(-0.75,1);
 \draw[dotted, thick] (0.75,-0)--(0.75,1);
 \draw (0.75,-0.4) node {$\frac{\delta}{2}$};
 \draw (-0.75,-0.4) node {$-\frac{\delta}{2}$};
 \draw[dotted, thick] (-2,-0)--(-2,1);
 \draw[dotted, thick] (2,-0)--(2,1);
 \draw (2,-0.4) node {$R$};
 \draw (-2,-0.4) node {$-R$};
 \draw[dotted, thick] (-4,-0)--(-4,1);
 \draw[dotted, thick] (4,-0)--(4,1);
 \draw (4,-0.4) node {$R_2$};
 \draw (-3.94,-0.4) node {$-R_2$};
 \draw[dotted, thick] (-4.5,-0)--(-4.5,1);
 \draw[dotted, thick] (4.5,-0)--(4.5,1);
 \draw (4.5,-0.4) node {$R_3$};
 \draw (-4.65,-0.4) node {$-R_3$};
\begin{scope}[yshift=-0.5cm]
 \draw (0.25,2) arc (0:180:0.25);
 \draw (0.75,2) arc (0:180:0.75);
 \draw[thin, dashed] (2,2) arc (0:180:2);
 \draw (4,2) arc (0:180:4);
 \draw (4.5,2) arc (0:180:4.5);
 \draw [black,fill=black!20!white] plot [smooth cycle, tension=1] coordinates {(0.7,2.5) (1.5,2.5)  (1.2,3) (0.4,3.9)};
 \draw (1, 2.8) node {$\Theta$};
\end{scope}
\end{tikzpicture}
\caption{Cutoff function $\eta$ and geometric setting of Theorem~\ref{thm:qUC}\label{fig:cutoff}. Here, $R_2 = 2 \mathrm{e} \vartheta_1 R$ and $R_3 = 2 \mathrm{e} \vartheta_1 R + 2D_0 = \rho$.}
\end{figure}
See Fig.~\ref{fig:cutoff} for an illustration of the geometric setting. Note that $\eta \equiv 1$ on $\Theta$ by assumption. Recall that $\epsilon_0 = 1 - 33 \mathrm{e} R d \vartheta_1^6 \vartheta_2$ and set
\begin{equation*}
\rho := 2 \mathrm{e} \vartheta_1 R + 2D_0
\quad \text{and} \quad
\mu := 33 d \rho \vartheta_1^{11/2} \vartheta_2 +  \frac{\rho \epsilon_0}{2 \mathrm{e} R \sqrt{\vartheta_1}}  ,
\end{equation*}
Since $\epsilon_0 > 0$ by assumption, we have $\mu > 33 d \rho \vartheta_1^{11/2} \vartheta_2$. Hence we can apply the Carleman estimate from Theorem~\ref{thm:carleman} with these choices of $\rho$ and $\mu$ to the function $u = \eta \psi$ and obtain for all $\alpha \geqslant \alpha_0 = \alpha_0 (d,\rho, \vartheta_1 , \vartheta_2 , \mu)$
\begin{equation*}
 I_1:= \int_{B(\rho)} \alpha^3 w^{-1-2\alpha} \lvert \eta \psi  \rvert^2  \leqslant \rho^4 C \int_{B(\rho)} w^{2-2\alpha} \lvert \Op(\eta \psi) \rvert^2  ,
\end{equation*}
where $C = C (d , \vartheta_1 , \rho \vartheta_2 , \mu) > 0$ and $w = w_{\rho , \mu}$. 
The Leibniz rule and $(a+b+c)^2 \leqslant 3(a^2+b^2+c^2)$ yields that $I_1$ is bounded by
\begin{align}
 I_1 & =  \rho^4 C \int_{B(\rho)} w^{2-2\alpha} \biggl\lvert \lvert \Op_c \eta \rvert \psi + (\Op \psi)\eta + 2 \sum_{i,j=1}^d a^{ij} (\partial_i \eta)(\partial_j \psi) \biggr\rvert^2  \nonumber \\
& \leqslant 3 \rho^4 C  \int_{B(\rho)} w^{2-2\alpha} \biggl( \lvert \Op_c \eta  \rvert^2 \lvert \psi \rvert^2 + \lvert \Op \psi \rvert^2 \eta^2 + 4 \Bigl \lvert \sum_{i,j=1}^d a^{ij} (\partial_i \eta)(\partial_j \psi) \Bigr \rvert^2 \biggr)   \label{eq:Leibnitz} ,
\end{align}
where $\Op_c \eta = -\diver (A \nabla \eta) + b^\T \nabla \eta$.
The pointwise estimate $\lvert \Op \psi \rvert \leqslant \lvert V \psi \rvert + \lvert \zeta \rvert$, $\lVert V \rVert_\infty \leqslant K_V$, and $w \leqslant \sqrt{\vartheta_1}$ on $B(\rho)$ gives
\begin{equation} \label{eq:subsolution}
 \int_{B(\rho)} w^{2-2\alpha} \lvert \Op \psi \rvert^2\eta^2  \leqslant   2 K_V^2 \vartheta_1^{3/2} \int_{B (\rho)} w^{-1-2\alpha}  \lvert \eta \psi \rvert^2  +
  2 \int_{B(\rho)} w^{2-2\alpha}  \lvert \eta \zeta \rvert^2 .
\end{equation}
From Ineq.~\eqref{eq:Leibnitz} and Ineq.~\eqref{eq:subsolution} we obtain for all $\alpha \geqslant \alpha_0$
\begin{multline*}
 \left[ \frac{\alpha^3}{3 \rho^4 C} - 2 K_V^2 \vartheta_1^{3/2}  \right] \int_{B(\rho)} \left( w^{-1-2\alpha} \lvert \eta \psi \rvert^2  \right) \\
\leqslant   \int_{B(\rho)} w^{2-2\alpha} \Biggl( \lvert \Op_c \eta \rvert^2 \lvert \psi \rvert^2 + 4 \Bigl\lvert \sum_{i,j=1}^d a^{ij} (\partial_i \eta)(\partial_j \psi) \Bigr\rvert^2 + 2 \lvert \eta \zeta \rvert^2 \Biggr)  =: I_2 + 2 \int_{B(\rho)} w^{2-2\alpha} \lvert \eta \zeta \rvert^2  .
\end{multline*}
(Hence, we have subsumed one term from the right hand side of the Carleman estimate in the left hand side. This was only possible because we did not apply Ineq.~\eqref{eq:weightbound2} to this term yet.)
Note that
\begin{equation}\label{eq:later}
2 \int_{B (\rho)} w^{2-2\alpha} \lvert \eta \zeta \rvert^2 \leqslant
2 \left(\frac{4\rho  \mu_1 \sqrt{\vartheta_1}}{\delta} \right)^{2\alpha - 2}
  \lVert \zeta \rVert^2_{\Omega} ,
\end{equation}
which will be used later.
Additionally to $\alpha \geqslant \alpha_0$ we choose
\begin{equation*} 
 \alpha \geqslant \sqrt[3]{16 \rho^4 C K_V^2 \vartheta_1^{3/2}} =:\alpha_1 .
\end{equation*}
This ensures that
\[
I_2 + 2 \int_{B (\rho)} w^{2-2\alpha} \lvert \eta \zeta \rvert^2  \geqslant \frac{5}{24} \frac{\alpha^3 }{\rho^4 C }  \int_{B (\rho)} \left( w^{-1-2\alpha} \lvert \eta \psi \rvert^2  \right)   .
\]
Since $\eta \equiv 1$ on $\Theta$ and by our bound on the weight function we have
\begin{equation} \label{eq:I2lower}
 I_2 + 2 \int_{B (\rho)} w^{2-2\alpha} \lvert \eta \zeta \rvert^2  \geqslant \frac{5}{24} \frac{\alpha^3}{\rho^4 C}  \left( \frac{\rho}{\sqrt{\vartheta_1} R} \right)^{1+2\alpha} \lVert \psi \rVert_\Theta^2 .
\end{equation}
Now we turn to an upper bound on $I_2$.
Since $a^{ij} = a^{ji}$ and $A (x) = (a^{ij} (x))_{i,j=1}^d$ is positive definite for all $x \in B(\rho)$, we can apply Cauchy Schwarz and obtain
 \begin{align*}
\Bigl\lvert \sum_{i,j=1}^d a^{ij} (\partial_i \eta)(\partial_j \psi) \Bigr \rvert^2 &\leqslant \Bigl( \sum_{i,j=1}^d a^{ij} (\partial_i \eta)(\partial_j \eta) \Bigr) \Bigl( \sum_{i,j=1}^d a^{ij} (\partial_i \overline{\psi})(\partial_j \psi) \Bigr)
 {\leqslant \vartheta_1 \lvert \nabla\eta\rvert^2(\nabla\psi^\T A \nabla\overline{\psi})}
  .
\end{align*}
Since $\Op_c \eta \not = 0$ only on $\supp \nabla \eta$ we have
\[
  I_2 \leqslant \int_{\supp \nabla \eta} w^{2-2\alpha} \Biggl( \lvert \Op_c \eta \rvert^2 \lvert \psi \rvert^2 + 4 {\vartheta_1 \lvert \nabla\eta\rvert^2 (\nabla\psi^\T A \nabla\overline{\psi})}  \Biggr)   .
\]
We split the integral according to the two components
\[
 B_1 = \left\{ x \in \mathbb{R}^d \colon \frac{\delta}{4} \leqslant \lvert x \rvert \leqslant \frac{\delta}{2} \right\}
 \quad \text{and} \quad
 B_2 = \left\{ x \in \mathbb{R}^d \colon 2 \mathrm{e} \vartheta_1 R \leqslant \lvert x \rvert \leqslant 2 \mathrm{e} \vartheta_1 R + D_0 \right\}
\]
of $\supp \nabla \eta$ and obtain by using the property~\eqref{eq:prop_eta} of the function $\eta$
\begin{multline*}
 I_2 \leqslant \int_{B_1} w^{2-2\alpha}
       \left( (\Op_c \eta)^2 \lvert \psi \rvert^2 + 4 {\vartheta_1} \left( \frac{M}{\delta} \right)^4 {\nabla\psi^\T A \nabla\overline{\psi}} \right)
       \\
 +  \int_{B_2} w^{2-2\alpha} \Biggl( (\Op_c \eta)^2 \lvert \psi \rvert^2 + 4 {\vartheta_1} \left( \frac{M}{D_0} \right)^4 {\nabla\psi^\T A \nabla\overline{\psi}} \Biggr)   .
\end{multline*}
On $B_1$ we use the general bound \eqref{eq:weightbound2} on the weight function $w$. In order to estimate the weight function on $B_2$, we note that for $r \geqslant 1/\mu$ we have by using $1 - \mathrm{e}^{-\mu t} \leqslant \mu t$ and $1 - \mathrm{e}^{-\mu t} < 1$ the bound
\[
 \varphi (r) \geqslant r \exp \left( -\int_0^{1/\mu} \mu \mathrm{d} t \right) \exp \left( -\int_{1/\mu}^r \frac{1}{t} \mathrm{d} t \right) = \frac{1}{\mathrm{e} \mu} .
\]
Since $\sigma (x) \geqslant \vartheta_1^{-1/2} \lvert x \rvert$, this implies that for for all $x \in B(\rho)$ with $\lvert x \rvert \geqslant \sqrt{\vartheta_1} \rho / \mu$
 we have
\begin{equation} \label{eq:weightbound3}
 w (x) = \varphi (\sigma (x/\rho)) \geqslant \frac{1}{\mathrm{e} \mu} .
\end{equation}
Since $2 \mathrm{e} \vartheta_1 R \geqslant \sqrt{\vartheta_1} \rho / \mu$ by our choice of $\mu$, we can use the bound \eqref{eq:weightbound3} for all $x \in B_2$. Hence, for $\alpha \geqslant 1 =: \alpha_2$ we arrive at
\begin{multline*}
 I_2 \leqslant \left( \frac{4 \rho \mu_1 \sqrt{\vartheta_1}}{\delta} \right)^{2 \alpha - 2} \int_{B_1}
       \left( (\Op_c \eta)^2 \lvert \psi \rvert^2 + 4 {\vartheta_1} \left( \frac{M}{\delta} \right)^4 {\nabla\psi^\T A \nabla\overline{\psi}} \right)
       \\
 + \left( \mathrm{e} \mu \right)^{2 \alpha - 2} \int_{B_2} \Biggl( (\Op_c \eta)^2 \lvert \psi \rvert^2 + 4 {\vartheta_1} \left( \frac{M}{D_0} \right)^4 {\nabla\psi^\T A \nabla\overline{\psi}}  \Biggr)   .
\end{multline*}
Now we use the pointwise estimate
\[
 \lvert \Op_c \eta \rvert^2
 \leqslant  3\vartheta_1^2 \lvert \Delta \eta \rvert^2
 +     3 \vartheta_1^2 (2d - 1)^2 \frac{\lvert \nabla \eta \rvert^2}{\lvert x \rvert^2}
 +    3 (\vartheta_2 d^2 + \lVert b \rVert_\infty)^2 \lvert \nabla \eta \rvert^2 ,
\]
see Appendix~\ref{sec:pointwise}, and obtain again by using the property~\eqref{eq:prop_eta} of the function $\eta$ that $I_2$ is bounded from above by
\begin{multline*}
 \left(\frac{4 \rho \mu_1 \sqrt{\vartheta_1}}{\delta}\right)^{2\alpha - 2}
 \left( \frac{M}{\delta}\right)^4
 \int_{B_1}
       \left[  \left(3 \vartheta_1^2 + \frac{768 \vartheta_1^2 d^2}{\delta^2}  + 3  (\vartheta_2 d^2 + \lVert b \rVert_\infty )^2 \right) \lvert \psi \rvert^2 + 4{\vartheta_1}  {\nabla\psi^\T A \nabla\overline{\psi}} \right]
       \\
 + \left( \mathrm{e} \mu \right)^{2\alpha - 2} \left( \frac{M}{D_0} \right)^4
 \int_{B_2}  \left[  \left(3 \vartheta_1^2 +  \frac{12 \vartheta_1^2 d^2}{(2 \mathrm{e} \vartheta_1 R)^2} + 3 (\vartheta_2 d^2 + \lVert b \rVert_\infty )^2\right) \lvert \psi \rvert^2 + 4 {\vartheta_1}  {\nabla\psi^\T A \nabla\overline{\psi}}  \right]  .
\end{multline*}
An application of Lemma~\ref{lemma:Cacciopoli} with $r_1 = \delta / 4$, $r_2 = \delta / 2$ and $r = \delta / 2$ for the first summand and $r_1 = 2 \mathrm{e} \vartheta_1  R$, $r_2 = 2 \mathrm{e} \vartheta_1  R + D_0$ and $r = D_0 / 2$ for the second summand gives by using $D_0 \geqslant \delta$ and $2 \mathrm{e} \vartheta_1 R \geqslant \delta$
\begin{multline} \label{eq:I2upper}
 I_2 \leqslant
       \left(\frac{4 \rho \mu_1 \sqrt{\vartheta_1}}{\delta}\right)^{2\alpha - 2} \left( \frac{M}{\delta}\right)^4
       \left[ 3 \vartheta_1^2 + \frac{768 \vartheta_1^2 d^2}{\delta^2} + 3(\vartheta_2 d^2 + \lVert b \rVert_\infty )^2   + 4{\vartheta_1}  \Cac_{\delta / 2} \right]       \lVert \psi \rVert_{B(\delta)}^2 \\
 +  \left( \mathrm{e} \mu \right)^{2\alpha - 2} \left( \frac{M}{D_0} \right)^4 \left[ 3 \vartheta_1^2 + \frac{3\vartheta_1^2 d^2}{(2 \mathrm{e} \vartheta_1 R)^2} + 3 (\vartheta_2 d^2 + \lVert b \rVert_\infty )^2  + 4 {\vartheta_1}  \Cac_{D_0 / 2}  \right] \lVert \psi \rVert_\Omega^2  \\
 +  \left(\frac{4\rho \mu_1 \sqrt{\vartheta_1}}{\delta} \right)^{2\alpha - 2}\left( \frac{M}{\delta} \right)^4
 16{\vartheta_1} \lVert \zeta \rVert^2_{\Omega} .
\end{multline}
where $\Cac_{t} = \Cac_t (V , b , c , \vartheta_1)$, $t > 0$, is defined in Lemma~\ref{lemma:Cacciopoli}.
Next we want to use
\begin{multline*}
 \left( \mathrm{e} \mu \right)^{2\alpha - 2} \left( \frac{M}{D_0} \right)^4 \left[ 3 \vartheta_1^2 + \frac{3\vartheta_1^2 d^2}{(2 \mathrm{e} \vartheta_1 R)^2} + 3 (\vartheta_2 d^2 + \lVert b \rVert_\infty )^2  + 4 {\vartheta_1}  \Cac_{D_0/2}  \right] \lVert \psi \rVert_\Omega^2 \\ \leqslant  \frac{3}{24} \frac{\alpha^3}{C \rho^4}  \left( \frac{\rho}{\sqrt{\vartheta_1} R } \right)^{1+2\alpha} \lVert \psi \rVert_\Theta^2  .
\end{multline*}
by choosing $\alpha$ sufficiently large. Since $\lVert \psi \rVert_\Omega^2 / \lVert \psi \rVert_\Theta^2 \leqslant \beta$, this is satisfied if
\begin{multline} \label{eq:alpha2}
 \alpha^3  \left( \frac{\rho}{\sqrt{\vartheta_1} \mathrm{e} R \mu} \right)^{2\alpha} \\  \geqslant
 \frac{8 C \rho^3 \sqrt{\vartheta_1} R \beta}{\mathrm{e}^2 \mu^2}
 \left( \frac{M}{D_0} \right)^4
 \left[ 3 \vartheta_1^2 + \frac{3\vartheta_1^2 d^2}{(2 \mathrm{e} \vartheta_1 R)^2} + 3 (\vartheta_2 d^2 + \lVert b \rVert_\infty )^2  + 4 {\vartheta_1}  \Cac_{D_0 / 2}  \right].
\end{multline}
Since we want to verify Ineq.~\eqref{eq:alpha2} by choosing $\alpha$ sufficiently large, we now argue that $\rho / (\sqrt{\vartheta_1} \mathrm{e} \allowbreak R \mu) > 1$. Indeed, by our assumption $\epsilon_0 = 1 - 33 \mathrm{e} d \vartheta_1^6 \vartheta_2 R > 0$ we have
\[
 \mu = 33 d \vartheta_1^{11/2} \vartheta_2 \rho + \frac{\rho \epsilon_0}{2 \sqrt{\vartheta_1} \mathrm{e} R} = \frac{33}{2} d \vartheta_1^{11/2} \vartheta_2 \rho +  \frac{\rho}{2 \sqrt{\vartheta_1} \mathrm{e} R} <  \frac{\rho}{\sqrt{\vartheta_1} \mathrm{e} R} .
\]
Hence,
\[
 \frac{\rho}{\sqrt{\vartheta_1} \mathrm{e} R \mu} > 1 ,
\]
which ensures that there exists $\alpha_3$ such an Ineq.~\eqref{eq:alpha2} is satisfied for all $\alpha \geqslant \max\{ 1, \alpha_3\}$. A possible choice of $\alpha_3$ is
\begin{multline}
 \alpha_3 :=  \left(2 \ln \left( \frac{\rho}{\sqrt{\vartheta_1} \mathrm{e} R \mu}  \right) \right)^{-1}  \ln \left( \frac{8 C \rho^3 \sqrt{\vartheta_1} R \beta}{\mathrm{e}^2 \mu^2}
 \left( \frac{M}{D_0} \right)^4  \right.
 \\
  \left. 
  \cdot
  \left[ 3 \vartheta_1^2 + \frac{3\vartheta_1^2 d^2}{(2 \mathrm{e} \vartheta_1 R)^2} + 3( \vartheta_2 d^2 + \lVert b \rVert_\infty )^2 + 4 {\vartheta_1} \Cac_{D_0 / 2}  \right] \right) .
\end{multline}
Hence, 
we can subsume the second summand of the right hand side of Ineq.~\eqref{eq:I2upper} into the lower bound of Ineq.~\eqref{eq:I2lower} and obtain by using Ineq.~\eqref{eq:later} and $M \geqslant 1$ for all $\alpha \geqslant \max\{\alpha_0 , \alpha_1 , \alpha_2 , \alpha_3\}$
\begin{multline*} 
 \frac{1}{12} \frac{\alpha^3}{C \rho^4}  \left( \frac{\rho}{\sqrt{\vartheta_1} R} \right)^{1+2\alpha} \lVert \psi \rVert_\Theta^2
 \\
 \leqslant
 \left(\frac{4 \rho \mu_1 \sqrt{\vartheta_1}}{\delta}\right)^{2\alpha - 2} \left( \frac{M}{\delta}\right)^4
       \left[ 3 \vartheta_1^2 + \frac{768 \vartheta_1^2 d^2}{\delta^2} + 3(\vartheta_2 d^2 + \lVert b \rVert_\infty )^2   + 4{\vartheta_1} \Cac_{\delta / 2} \right]       \lVert \psi \rVert_{B(\delta)}^2 \\
       + \left(\frac{4\rho \mu_1 \sqrt{\vartheta_1}}{\delta} \right)^{2\alpha - 2}\left( \frac{M}{\delta} \right)^4
 18\vartheta_1^2 \lVert \zeta \rVert^2_{\Omega}  .
\end{multline*}
Now we use the lower bound $\alpha^3 \geqslant 1$ and
\[
 \frac{18\vartheta_1^2}{ 3 \vartheta_1^2 + \frac{768 \vartheta_1^2 d^2}{\delta^2} + 3(\vartheta_2 d^2 + \lVert b \rVert_\infty )^2   + 4{\vartheta_1} \Cac_{\delta / 2} } \leqslant \delta^2 ,
\]
and obtain the statement of the theorem with
\begin{equation} \label{eq:constant_qUC}
 C_\qUC := \frac{4 \mu_1^2 \sqrt{\vartheta_1} \delta^{2} (3 R \rho C M^4 )^{-1}}{3 \vartheta_1^2 + 768 \vartheta_1^2 d^2 \delta^{-2} + 3(\vartheta_2 d^2 + \lVert b \rVert_\infty )^2   + 4{\vartheta_1}  \Cac_{\delta / 2}} \left( \frac{\delta}{4 \mu_1  \vartheta_1 R}  \right)^{2\alpha^*}
\end{equation}
where $\alpha^* := \max\{\alpha_0 , \alpha_1 , \alpha_2 , \alpha_3 \}$.
\end{proof}
%
%
%
%
%
\subsection{Proof of Theorem~\ref{thm:sampling} and Theorem~\ref{cor:sampling}}
We follow \cite[Proof of Theorem~2.1]{Rojas-MolinaV-13}.
For $L > 0$ and $x \in \mathbb{R}^d$ we denote by $\Lambda_L(x) = (-L/2 , L/2)^d + x$ the cube of side length $L$ centered at $x$, and for $a \in \mathbb{R}$ by $\lceil a\rceil$ the smallest integer larger than or equal to $a$.
\begin{proof}[Proof of Theorem~\ref{thm:sampling}]
Fix $\psi \in W^{2,2} (\mathbb{R}^d)$ and $\zeta \in L^2 (\mathbb{R}^d)$. We say that a site $k \in \ZZ^d$ is dominating if
\begin{equation}\label{2.3.1}
 \int_{\Lambda_1 (k)} \lvert \psi \rvert^2 \geqslant \frac{1}{2T^d} \int_{\Lambda_T (k)} \lvert \psi \rvert^2 \quad \text{with} \quad T = \left\lceil 2 (\sqrt{d} + 2) \left(2\mathrm{e} \vartheta_1 + 1 \right) \right\rceil,
\end{equation}
and otherwise weak. We denote by $W \subset \mathbb{R}^d$ the union of unit cubes centered at weak sites and by $D \subset \mathbb{R}^d$ the union of unit cubes centered at dominating sites. Then
\begin{align}
 \int_W \lvert \psi \rvert^2 = \sum_{\genfrac{}{}{0pt}{2}{k \in \ZZ^d:}{\text{$k$ is weak}}} \int_{\Lambda_1 (k)} \lvert \psi \rvert^2
 &<
 \frac{1}{2T^d} \sum_{\genfrac{}{}{0pt}{2}{k \in \ZZ^d:}{\text{$k$ is weak}}} \int_{\Lambda_T (k)} \lvert \psi \rvert^2 \nonumber \\
 &\leqslant
  \frac{1}{2T^d} \sum_{k\in\ZZ^d} \int_{\Lambda_T (k)} \lvert \psi \rvert^2
 \leqslant
 \frac{T^d}{2T^d} \int_{\mathbb{R}^d} \lvert \psi \rvert^2
 =
 \frac{1}{2} \int_{\mathbb{R}^d} \lvert \psi \rvert^2 . \label{2.3.2}
\end{align}
Since $D$ is the complement of $W$ in $\mathbb{R}^d$, we have
\begin{equation}\label{2.3.3}
 2 \int_D \lvert \psi \rvert^2 > \int_{\mathbb{R}^d} \lvert \psi \rvert^2 .
\end{equation}
For a dominating site $k \in \ZZ^d$ we define its right near-neighbor by
\[
 k^+ = k+ 2 \mathbf{e}_1
 \quad \text{where} \quad
 \mathbf{e}_1 = (1,0,\ldots,0)^{\mathrm{T}} \in \mathbb{R}^d.
\]
For each dominating site $k \in \ZZ^d$ we want to apply Theorem~\ref{thm:qUC} with
\begin{align*}
\beta = 2T^d, \quad
 R = \sqrt{d} + 2, \quad
 D_0 = \frac{R}{2}, \quad
 \Omega = \Lambda_T (k), \quad
 x = z_{k^+}  \quad \text{and} \quad
 \Theta = \Lambda_1 (k) .
\end{align*}
Therefore, we have to check whether the assumptions of Theorem~\ref{thm:qUC} are satisfied for these specific choices. Assumption~\eqref{ass:qUC} of Theorem~\ref{thm:qUC} is satisfied by Assumption~\eqref{ass:samplingG=1}. Note that $\Theta =  \Lambda_1(k)$ is obviously disjoint from the open ball $B(z_{k^+},\delta / 2)$, and there exists 
$a \in \Lambda_1(k)$ with $\lvert a-z_{k^+} \rvert_2=2$. Thus, for each $b \in \Lambda_1(k)$ we have
$\lvert b-z_{k^+}\rvert_2\leqslant \lvert b-a \rvert_2 + \lvert a-z_{k^+}\rvert_2\leqslant \sqrt{d}+ 2 = R$. Hence,
$\Theta \subset \overline{B (z_{k^+},R)} \setminus B (z_{k^+},\delta / 2)$. In order to verify $B (x , 2 \mathrm{e} \vartheta_1 R + 2D_0) \subset \Omega$,
we note that for each $y \in B (z_{k^+} , 2 \mathrm{e} \vartheta_1 R + 2D_0)$ we have
\begin{align*}
 \lvert k - y \lvert &\leqslant \lvert k-k^+  \rvert + \lvert k^+ - z_{k^+} \rvert + \lvert z_{k^+} - y  \rvert < 2 + \sqrt{d}/2 + (a+1)R
  \leqslant  T / 2 ,
\end{align*}
and therefore, $B (z_{k^+} , 2 \mathrm{e} \vartheta_1 R + 2D_0) \subset B (k , T/2) \subset \Lambda_T (k)$.  Finally, $\lVert \psi \rVert^2_\Omega / \lVert \psi \rVert_\Theta^2 \leqslant \beta$ since $k$ is dominating.
Thus, for every dominating site $k$ Theorem~\ref{thm:qUC} gives
\[
 \lVert \psi \rVert_{B (z_{k^+} , \delta)}^2 + \delta^2 \lVert \zeta \rVert_{\Omega}^2  \geqslant C_{\qUC}  \lVert \psi \rVert_{\Lambda_1 (k)}^2
\]
where $C_{\qUC}$ is given in Theorem~\ref{thm:qUC} with the above choices of the parameters. Since $B (z_{k^+},\delta) \subset \Lambda_1 (k^+)$ for all $k \in \ZZ^d$, we obtain by summing over all dominating $k \in \ZZ^d$
\begin{equation*}
\lVert \psi \rVert_{\equi_\delta}^2 + T^d \delta^2 \lVert \zeta \rVert_{\mathbb{R}^d}^2
\geqslant C_{\qUC} \lVert \psi \rVert_D^2 > \frac{C_{\qUC}}{2} \lVert \psi \rVert_{\mathbb{R}^d}^2 .
\end{equation*}
From Lemma~\ref{lemma:constants} we infer a lower bound on $C_{\qUC} / (2T^d)$ leading to the stated constant $c_{\sfUC}$.
\end{proof}
The proof of Theorem~\ref{thm:samplingG} is postponed to Appendix~\ref{app:scaling}.
\begin{proof}[Proof of Theorem~\ref{cor:sampling}]
We follow \cite[Proof of Theorem~1.1]{Klein-13}. Set $V \equiv E$ and $\zeta = (H - E) \psi$. Then the assumption $\lvert H \psi \rvert \leqslant \lvert E \psi \rvert + \lvert (H - E) \psi \rvert$ of Theorem~\ref{thm:samplingG}
is satisfied by the triangle inequality.
By using $\lVert (H-E) \psi \rVert^2 \leqslant \gamma^2 \lVert \psi \rVert^2$, we obtain the inequality
\begin{equation*}
 C_\sfUC \lVert \psi \rVert_{\mathbb{R}^d}^2
 \leqslant
 \lVert \psi \rVert_{\equi_{\delta}}^2
 +
 \delta^2 G^2 \lVert (H- E) \psi \rVert_{\mathbb{R}^d}^2
  \leqslant
 \lVert  \psi \rVert_{\equi_\delta}^2
 +
 \delta^2 G^2 \gamma^2 \lVert \psi \rVert_{\mathbb{R}^d}^2 .
\end{equation*}
The result follows, since $\delta < G/2$ and $C_\sfUC - \delta^2 G^2 \gamma^2 \geqslant (3/4) C_\sfUC$.
\end{proof}
\subsection{Proof of Theorem~\ref{thm:L:per} to \ref{cor:L:Dir}}
\label{sec:proofL}
\begin{proof}[Proof of Theorem~\ref{thm:L:per} and \ref{thm:L:Dir}]
Here we follow the main lines of the proof of Theorem~\ref{thm:sampling} up to certain minor changes.
This is due to the fact that the function $\psi$ is defined just on $\Lambda_L$ instead of $\mathbb{R}^d$.
We assume $G = 1$. The general case $G > 0$ follows by scaling, similar as explained in Appendix~\ref{app:scaling}.
We consider the differential expression
\[
 \Op_L := -\diver (A \nabla u) + b^\T \nabla u + c u = -\sum_{i,j=1^d} \partial_i \left( a^{ij} \partial_j u \right) + \sum_{i=1}^d b_i \partial_i u + c u ,
\]
with coefficient functions $A = (a^{ij})_{i,j=1}^d : \Lambda_L \to \mathbb{R}^{d \times d}$, $b : \Lambda_L \to \mathbb{C}^d$, and $c : \Lambda_L \to \mathbb{C}$.
Our assumption on $\psi$ is equivalent to $\lvert \Op_L \psi \rvert \leqslant \lvert V \psi \rvert + \lvert \zeta \rvert$ almost everywhere on $\Lambda_L$.
We want to extend $\psi$, $V$, $\zeta$ and the coefficients of $\Op_L$ in such a way, that the same inequality holds almost everywhere on $\mathbb{R}^d$.
\par
In the case of Theorem~\ref{thm:L:per}, i.e.\ periodic boundary conditions, we extend the function $\psi$ $L$-periodically in each direction. Namely, let $\mathbf{e}_p= (0,\ldots,0,1,0,\ldots,0)^{\mathrm{T}} \in \mathbb{R}^d$, $p=1,\ldots,d$, be the standard basis in $\mathbb{R}^d$, then
\begin{equation*}
\psi(x+L\mathbf{e}_p)=\psi(x).
\end{equation*}
In the same way we extend all the coefficients of $\Op_L$, the potential $V$ and the function $\zeta$.
\par
In the case of Theorem~\ref{thm:L:Dir}, i.e.\ Dirichlet boundary conditions, the extensions are different and made by symmetric and antisymmetric reflections with respect to the sides of $\Lambda_L$.
Namely, assume that functions $\psi$, $V$, $\zeta$, $a^{ij}$, $b_i$, $c$ are defined on $\Lambda_L(m)$, $m=(m_1,\ldots,m_d)\in L \ZZ^d$. Then the above functions are extended on the neighboring boxes $\Lambda_L(m\pm L\mathbf{e}_p)$, $p=1,\ldots,d$, as follows:
\begin{align*}
&\psi\big(x\pm L\mathbf{e}_p\big)=-\psi\big(x+2(m_p - x_p)\mathbf{e}_p\big),
\\
&a^{ij}\big(x\pm L\mathbf{e}_p\big)=a^{ij}\big(x+2(m_p -x_p)\mathbf{e}_p\big)\quad\text{if}\quad i\not=p,\quad j\not=p,
\\
&a^{pp}\big(x\pm L\mathbf{e}_p\big)=a^{pp}\big(x+2(m_p -x_p)\mathbf{e}_p\big),
\\
&a^{pj}\big(x\pm L\mathbf{e}_p\big)=a^{jp}\big(x\pm L\mathbf{e}_p\big)= -a^{pj}\big(x+2(m_p -x_p)\mathbf{e}_p\big)\quad \text{if}\quad p\not=j,
\\
&b_i\big(x\pm L\mathbf{e}_p\big)=-b_i\big(x+2(m_p -x_p)\mathbf{e}_p\big) \quad\text{if}\quad i\not=p,
\\
&b_p\big(x\pm L\mathbf{e}_p\big)=b_p\big(x+2(m_p -x_p)\mathbf{e}_p\big),
\\
& c\big(x\pm L\mathbf{e}_p\big)=c\big(x+2(m_p -x_p)\mathbf{e}_p\big),
\\
&V\big(x\pm L\mathbf{e}_p\big)=V\big(x+2(m_p -x_p)\mathbf{e}_p\big),
\\
&\zeta\big(x\pm L\mathbf{e}_p\big)=\zeta\big(x+2(m_p -x_p)\mathbf{e}_p\big) .
\end{align*}
Due to the assumption (Per) or (Dir)  and the corresponding boundary condition of $\psi$, the extended $\psi$ is locally in $W^{2,2} (\mathbb{R}^d)$, satisfies $\lvert \Op_L \psi \rvert \leqslant \lvert V \psi \rvert + \lvert \zeta \rvert$ almost everywhere on $\mathbb{R}^d$, and the coefficients of $\Op$ satisfy the ellipticity and Lipschitz condition \eqref{eq:elliptic}. Moreover, by construction of our extension of $\psi$ we have
\begin{equation}\label{2.4.1}
\sum_{k \in \Lambda_L \cap \ZZ^d} \int_{\Lambda_T (k)} \lvert \psi \rvert^2  =T^d \int_{\Lambda_L} \lvert \psi \rvert^2  .
\end{equation}
Now we proceed as in the proof of Theorem~\ref{thm:sampling}. We define dominating sites as in \eqref{2.3.1}. Instead of $W$, we introduce set $W(L)$ as the union of all unit cubes centered at weak sites $k$ located inside $\Lambda_L$.
By $D(L)$ we denote the union of unit cubes centered at the dominating sites located inside $\Lambda_L$. Then employing \eqref{2.4.1}, in the same way as in \eqref{2.3.2} we get
\begin{align*}
 \int_{W(L)} \lvert \psi \rvert^2  =  \sum_{\genfrac{}{}{0pt}{2}{k \in \ZZ^d\cap\Lambda_L:}{\text{$k$ is weak}}} \int_{\Lambda_1 (k)} \lvert \psi \rvert^2
 &<
 \frac{1}{2T^d} \sum_{\genfrac{}{}{0pt}{2}{k \in \Lambda_L\cap\ZZ^d:}{\text{$k$ is weak}}} \int_{\Lambda_T (k)} \lvert \psi \rvert^2
 \\
 &\leqslant
  \frac{1}{2T^d} \sum_{ k \in \Lambda_L\cap\ZZ^d} \int_{\Lambda_T (k)} \lvert \psi \rvert^2
 =
 \frac{1}{2}\int_{\Lambda_L} \lvert\psi\rvert^2.
\end{align*}
Hence, we arrive at the analogue of estimate \eqref{2.3.3}:
\begin{equation*}
2\int_{D(L)} \lvert \psi \rvert^2 > \int_{\Lambda_L} \lvert \psi \rvert^2 .
\end{equation*}
Now all other arguments in the proof of Theorem~\ref{thm:sampling} are reproduced literally and we obtain for all dominating sites $k \in \Lambda_L \cap \ZZ^d$
\[
 \lVert \psi \rVert_{B (z_{k^+} , \delta)} + \delta^2 \lVert \zeta \rVert_{\Lambda_T (k)}^2 \geqslant C_{\qUC} \lVert \psi \rVert_{\Lambda_1 (k)}^2 ,
\]
where $C_{\qUC}$ is the constant from Theorem~\ref{thm:qUC} with $\beta = 2T^d$, $R = \sqrt{d} + 2$, and $D_0 = R / 2$. By summing over all dominating sites $k \in \Lambda_L \cap \ZZ^d$ and thanks to the oddness of $L$ and \eqref{2.4.1} we obtain
\begin{equation*}
\lVert \psi \rVert_{\equi_{\delta,L}}^2 + \delta^2 T^d \lVert \zeta \rVert_{\Lambda_T (k)}^2
\geqslant \!\!\!\!\!\! \sum_{\genfrac{}{}{0pt}{2}{k \in \ZZ^d\cap \Lambda_L:}{\text{$k$ is dominating}}} \!\!\!\!\!\! \Bigl( \lVert \psi \rVert_{B (z_{k^+} , \delta)}^2 + \delta^2 \lVert \zeta \rVert_{\Lambda_T (k)}^2 \Bigr)
\geqslant  C_{\qUC}
\lVert \psi \rVert_{D(L)}^2 > \frac{C_{\qUC}}{2}
\lVert \psi \rVert_{\Lambda_L}^2.
\end{equation*}
From Lemma~\ref{lemma:constants} we infer a lower bound on $C_{\qUC} / (2T^d)$ leading to the stated constant $C_{\sfUC}$.
\end{proof}
\begin{proof}[Proof of Theorem~\ref{cor:L:per} and \ref{cor:L:Dir}]
We follow \cite[Proof of Theorem~1.1]{Klein-13}. Depending on the boundary condition $\bullet \in \{\mathrm{per} , \mathrm{Dir}\}$, we set $V \equiv E$ and $\zeta = (H_L^\bullet \psi -E) \psi$.
Then the assumption $\lvert \Op \psi \rvert \leqslant \lvert E \psi \rvert + \lvert \zeta \rvert$ of Theorem~\ref{thm:L:per} and \ref{thm:L:Dir} is satisfied by the  triangle inequality.
By using $\lVert (H_L^\bullet-E) \psi \rVert^2 \leqslant \gamma^2 \lVert \psi \rVert^2$, we obtain the inequality
\begin{equation*}
 C_\sfUC \lVert \psi \rVert_{\Lambda_L}^2
 \leqslant
 \lVert \psi \rVert_{\equi_{\delta , L}}^2
 +
 (\delta / G)^2 G^2 \lVert (H - E )\psi \rVert_{\Lambda_L}^2
  \leqslant
 \lVert  \psi \rVert_{\equi_\delta}^2
 +
 (\delta / G)^2 G^2 \gamma^2 \lVert \psi \rVert_{\Lambda_L}^2 .
\end{equation*}
The result follows, since $\delta < G/2$ and $C_\sfUC - \delta^2 G^2 \gamma^2 \geqslant (3/4) C_\sfUC$.
\end{proof}
%
%
%
%
%
%
 \appendix
\section{An estimate for rotational symmetric functions}\label{sec:pointwise}
Let $\eta \in C_{\mathrm c}^\infty (\mathbb{R}^d)$, where $\eta = \zeta \circ \sigma$ with $\sigma (x) = \lvert x \rvert$ and $\zeta : \mathbb{R} \to [0,1]$ some profile function.
We want to estimate $(\Op \eta)^2$ pointwise. For the first and second derivatives we have
\[
\frac{\partial\eta}{\partial x_i} = \zeta' (\sigma) \frac{x_i}{\lvert x \rvert}, \quad \frac{\partial^2 \eta}{\partial x_i \partial x_j} = \zeta'' (\sigma) \frac{x_i x_j}{\lvert x \rvert^2} + \zeta' (\sigma) \left( \frac{\delta_{ij}}{\lvert x \rvert} - \frac{x_i x_j}{\lvert x \rvert^3} \right) .
\]
For the gradient and the Laplacian of $\eta$ there holds
\[
 \lvert \nabla \eta \rvert = \lvert \zeta' (\sigma) \rvert  \quad \text{and} \quad \Delta \eta = \zeta'' (\sigma) + \zeta' (\sigma) \frac{d-1}{\lvert x \rvert} .
\]
First we estimate $\lvert \Op_0 \eta \rvert := \lvert -\diver ( A \nabla \eta) \rvert$:
\begin{align*}
 \lvert \Op_0 \eta \rvert &\leqslant \left\lvert \sum_{i,j}  a^{ij} \left(  \zeta'' (\sigma) \frac{x_i x_j}{\lvert x \rvert^2} + \zeta' (\sigma) \left( \frac{\delta_{ij}}{\lvert x \rvert} - \frac{x_i x_j}{\lvert x \rvert^3} \right) \right) \right\rvert + \left\lvert \sum_{i,j} \left(\partial_i a^{ij} \right) \zeta' (\sigma) \frac{x_j}{\lvert x \rvert} \right\rvert \\
 &=: S_1 + S_2
\end{align*}
For the summand $S_2$ we have
\[
 S_2 \leqslant \vartheta_2 \lvert \zeta' (\sigma) \rvert \sum_{i,j}  \frac{\lvert x_j \rvert}{\lvert x \rvert} \leqslant \vartheta_2 d^2 \lvert \zeta' (\sigma) \rvert = \vartheta_2 d^2 \lvert \nabla \eta \rvert .
\]
For the first summand $S_1$ we have
\begin{align*}
 S_1 &\leqslant \left\lvert\sum_{i,j} a^{ij}  \zeta'' (\sigma) \frac{x_i x_j}{\lvert x \rvert^2} \right\rvert + \left\rvert \sum_{i,j} a^{ij} \zeta' (\sigma) \left( \frac{\delta_{ij}}{\lvert x \rvert} - \frac{x_i x_j}{\lvert x \rvert^3}\right) \right\rvert
\end{align*}
By Assumption~\ref{eq:elliptic} we have $\sum_{i,j} a^{ij} x_i x_j \leqslant \vartheta_1 \lvert x \rvert^2$. This gives
\begin{align*}
 S_1 & \leqslant \vartheta_1 \lvert \zeta'' (\sigma) \rvert + \frac{\lvert \zeta' (\sigma) \rvert}{\lvert x \rvert} \left\rvert \sum_{i,j} a^{ij}  \left( \delta_{ij} - \frac{x_i x_j}{\lvert x \rvert^2}\right) \right\rvert  .
\end{align*}
By Assumption~\ref{eq:elliptic} we have $\vartheta_1^{-1} \leqslant a^{ii} \leqslant \vartheta_1$ for all $i \in \{1,\ldots , d\}$ and $\vartheta_1^{-1} \leqslant \lvert x \rvert^{-2} \sum_{i,j} a^{ij} x_i x_j \leqslant \vartheta_1$. Hence,
\begin{align*}
 S_1 &\leqslant \vartheta_1 \lvert \zeta'' (\sigma) \rvert + \frac{\lvert \zeta' (\sigma) \rvert}{\lvert x \rvert}  d \vartheta_1 \\
 & = \vartheta_1  \left \lvert \Delta \eta - \zeta' (\sigma) \frac{d-1}{\lvert x \rvert} \right\rvert + \lvert \nabla \eta \rvert \frac{d \vartheta_1}{\lvert x \rvert} \\
 & \leqslant \vartheta_1  \vert \Delta \eta \rvert + \vartheta_1 \lvert \nabla \eta \rvert \frac{d-1}{\lvert x \rvert} + \lvert \nabla \eta \rvert \frac{d \vartheta_1}{\lvert x \rvert} \\
 &= \vartheta_1 \lvert \Delta \eta \rvert + \vartheta_1 \frac{\lvert \nabla \eta \rvert}{\lvert x \rvert} (2d-1) .
\end{align*}
Putting everything together, 
 we obtain
\[
 \lvert \Op_0 \eta \rvert
 \leqslant  \vartheta_1 \vert \Delta \eta \rvert
 +     \vartheta_1 (2d - 1) \frac{\lvert \nabla \eta \rvert}{\lvert x \rvert}
 +     \vartheta_2 d^2 \lvert \nabla \eta \rvert .
\]
Hence,
\[
 \lvert \Op_c \eta \rvert
 \leqslant \lvert \Op_0 \eta \rvert + \lvert b^\T \nabla \eta \rvert
 \leqslant  \vartheta_1 \vert \Delta \eta \rvert
 +     \vartheta_1 (2d - 1) \frac{\lvert \nabla \eta \rvert}{\lvert x \rvert}
 +     (\vartheta_2 d^2 + \lVert b \rVert_\infty) \lvert \nabla \eta \rvert
\]
and
\[
 \lvert \Op_c \eta \rvert^2
 \leqslant  3\vartheta_1^2 \lvert \Delta \eta \rvert^2
 +     3 \vartheta_1^2 (2d - 1)^2 \frac{\lvert \nabla \eta \rvert^2}{\lvert x \rvert^2}
 +    3 (\vartheta_2 d^2 + \lVert b \rVert_\infty)^2 \lvert \nabla \eta \rvert^2  .
\]
\section{The constant $C_{\qUC}$} \label{sec:CqUC}
We derive an explicit bound on $C_{\mathrm{qUC}}$ in the special case $2D_0 = R \geqslant 1$, and $\delta \leqslant 2$.
In this special case we have $\epsilon_0 \in (0,1]$, $\mu \in (\sqrt{\vartheta_1} , (5/2) \sqrt{\vartheta_1})$, $\mu_1 \in (1,(5/2) \mathrm{e}\vartheta_1)$, $\rho = (2 \mathrm{e} \vartheta_1+1)R \leqslant (2 \mathrm{e} + 1)\vartheta_1 R$, and $C_\mu \in ( \sqrt{\vartheta_1} \epsilon_0 , (3/2) \sqrt{\vartheta_1} \epsilon_0)$.
By $K$ we denote constants depending only on the dimension which may change from line to line. For the constant $C$ from the Carleman estimate we have the upper bound
\[
 C
 \leqslant
 K \epsilon_0^{-1} \vartheta_1^{13} \mathrm{e}^{10\vartheta_1} (1+\vartheta_2) R.
\]
For the constant $\Cac_{\delta / 2}$ we have by using $\delta \leqslant 2$ and $x \leqslant 1+x^2$
\[
  \Cac_{\delta / 2} \leqslant \frac{3 + 2 \vartheta_1 (\lVert V \rVert_\infty^2 + \lVert b \rVert_\infty^2 + \lVert c \rVert_\infty^2) + 8 \vartheta_1 C'}{\delta^2 / 4} \leqslant K \frac{\vartheta_1}{\delta^2} \left( 1 + \lVert V \rVert_\infty^2 + \lVert b \rVert_\infty^2 + \lVert c \rVert_\infty^2 \right).
 \]
 Hence, for the term
 \[
  T_1 = \frac{4 \mu_1^2 \sqrt{\vartheta_1} \delta^{2} (3 R \rho C M^4 )^{-1}}{3 \vartheta_1^2 + 768 \vartheta_1^2 d^2 \delta^{-2} + 3(\vartheta_2 d^2 + \lVert b \rVert_\infty )^2   + 4{\vartheta_1}  \Cac_{\delta / 2}}
 \]
 in $C_{\mathrm{qUC}}$, we obtain the lower bound
 \begin{align*}
  T_1 &\geqslant
       \frac{K\vartheta_1^{-31/2} \mathrm{e}^{-10 \sqrt{\vartheta_1}} \epsilon_0 \delta^4}{R^3 (1+\vartheta_2)^3 (1 + \lVert V \rVert_\infty^2 + \lVert b \rVert_\infty^2 + \lVert c \rVert_\infty^2)} .
 \end{align*}
 For the constant $C_{\mathrm{qUC}}$ we obtain using $\delta^4 / R^3 \geqslant (\delta / (10 \mathrm{e} \vartheta_1^{2} R))^4$, $(1+x)^{-1} = (\mathrm{e}^{-1})^{\ln (1+x)}$ and $\ln (1+x) \leqslant 3x^{1/3}$ with $x = \lVert V \rVert_\infty^2 + \lVert b \rVert_\infty^2 + \lVert c \rVert_\infty^2$
\begin{equation*}
 C_{\mathrm{qUC}}
      \geqslant  \frac{K\vartheta_1^{-31/2} \mathrm{e}^{-10\vartheta_1} \epsilon_0}{(1+\vartheta_2)^3}
      \left( \frac{\delta}{10 \mathrm{e} \vartheta_1^{2} R} \right)^{4+2\alpha^*+3(\lVert V \rVert_\infty^2 + \lVert b \rVert_\infty^2 + \lVert c \rVert_\infty^2)^{1/3}} .
\end{equation*}
Next we give an upper bound on $\alpha^*$. For $\tilde \alpha_0$ we have
\[
 \tilde \alpha_0 \leqslant K \vartheta_1^{25} \mathrm{e}^{15 \vartheta_1}(1+\vartheta_2)^2 R^2 \epsilon_0^{-1} .
\]
Hence,
\[
 \alpha_0 \leqslant K \vartheta_1^{25} \mathrm{e}^{15 \vartheta_1} (1+ \vartheta_2)^2 \epsilon_0^{-1} R^3 (1+ \lVert b \rVert_\infty^2 + \lVert c \rVert_\infty^{2/3})  .
\]
For $\alpha_1$ we have
 \[
  \alpha_1 \leqslant K
  \vartheta_1^{37/6} \mathrm{e}^{10 \vartheta_1 / 3} (1+\vartheta_2)^{1/3} \epsilon_0^{-1/3}
  \left( \lVert V \rVert_\infty^2 R^5 \right)^{1/3} .
 \]
Since
\[
 \frac{\sqrt{\vartheta_1} \mathrm{e} R \mu}{\rho} = 33 d \mathrm{e} R \vartheta_1^6 \vartheta_2 + \epsilon_0 / 2 = 1 - \epsilon_0 / 2
\]
we find
\[
 \ln \left( \frac{\rho}{\sqrt{\vartheta_1} \mathrm{e} R \mu} \right) = \ln \left( \frac{1}{1- \epsilon_0 / 2} \right) \geqslant \epsilon_0 / 2.
\]
Hence, using $F_{D_0 / 2} \leqslant K \vartheta_1^2 (1 + \lVert V \rVert_\infty^2 + \lVert b \rVert_\infty^2 + \lVert c \rVert_\infty^2)$ we obtain as an upper bound for $\alpha_3$
\[
 \alpha_3 \leqslant \epsilon_0^{-1} \ln \left(K \beta R {\vartheta_1^{19}} \mathrm{e}^{10 \vartheta_1} (1+\vartheta_2)^3 \epsilon_0^{-1} \left( \lVert V \rVert_\infty^2 + \lVert b \rVert_\infty^2 + \lVert c \rVert_\infty^2 + 1 \right) \right) .
\]
Hence,
\begin{multline*}
 \alpha^* \leqslant 1 + K \vartheta_1^{25} \mathrm{e}^{15 \vartheta_1} (1 + \vartheta_2)^2 \epsilon_0^{-1} \bigl(1 + \lVert b \rVert_\infty^{2} + \lVert c \rVert_\infty^{2/3} + \lVert V \rVert_\infty^{2 / 3} \bigr) R^{3} \\
 + \epsilon_0^{-1} \ln \left(K \beta R {\vartheta_1^{19}} \mathrm{e}^{10 \vartheta_1}  (1+\vartheta_2)^3 \epsilon_0^{-1} \left( \lVert V \rVert_\infty^2 + \lVert b \rVert_\infty^2 + \lVert c \rVert_\infty^2 + 1 \right) \right) .
\end{multline*}
We now use $\ln (1+x) \leqslant 3x^{1/3}$ for $x \geqslant 0$, $(\sum a_i)^{1/3} \leqslant \sum a_i^{1/3}$ and $\epsilon_0 \geqslant (\delta / (10\mathrm{e}\vartheta_1^{2} R))^{-\ln \epsilon_0}$, and obtain
\begin{align*}
 C_{\mathrm{qUC}} &\geqslant \frac{K \vartheta_1^{-31/2} \mathrm{e}^{-10\vartheta_1}}{(1+\vartheta_2)^3}
 \left( \frac{\delta}{10 \mathrm{e} \vartheta_1^{2} R} \right)^{\gamma}
\end{align*}
where
\begin{align*}
 \gamma &= K \epsilon_0^{-1}  \vartheta_1^{25} \mathrm{e}^{15 \vartheta_1} (1+\vartheta_2)^2
 \left(
 1 + \lVert V \rVert_\infty^{2/3} + \lVert b \rVert_\infty^{2} + \lVert c \rVert_\infty^{2/3} \right) R^3 + \ln \beta - \ln \epsilon_0   .
\end{align*}
\section{Scaling argument (deduction of Theorem~\ref{thm:samplingG} from Theorem~\ref{thm:sampling})} \label{app:scaling}
Since we proved Theorem~\ref{thm:samplingG}, \ref{thm:L:per} and \ref{thm:L:Dir} only in the special case $G = 1$, we show in this appendix how the general case $G > 0$ can be obtained by scaling. We restrict this discussion to Theorem~\ref{thm:samplingG} for functions on $\mathbb{R}^d$.
The argument applies in the same way to Theorem~\ref{thm:L:per} and \ref{thm:L:Dir} for functions on the cube $\Lambda_L$.
\par
We fix a $(G,\delta)$-equidistributed sequence with $G > 0$ and $\delta \in (0,G/2)$. Let $g:\mathbb{R}^d \to \mathbb{R}^d$ be given by $g (x) = G x$ and $\tilde \psi = \psi \circ g$. We want to estimate
\[
 \int_{S_\delta} \lvert \psi \rvert^2 = G^d \int_{S_\delta / G} \lvert \tilde \psi \rvert^2
\]
from below. Note that $S_\delta / G$ corresponds to some $(1,\delta / G)$-equidistributed sequence. Let $\tilde a^{ij} = a^{ij} \circ g$, $\tilde b = G (b \circ g)$, $\tilde c = G^2(c \circ g)$, $\tilde V = G^2 (V \circ g)$, $\tilde \zeta = G^2 (\zeta \circ g)$ and
\[
 \tilde \Op = \sum_{i,j=1}^d \partial_i \tilde a^{ij} \partial_j + \tilde b^\T \nabla + \tilde c .
\]
Then, by the chain rule and our assumption $\lvert \Op \psi \rvert \leqslant \lvert V \psi \rvert + \lvert \zeta \rvert$, we have almost everywhere on $\mathbb{R}^d$ the inequality
\[
 \lvert \tilde \Op \tilde \psi \rvert
=  G^2  \lvert (\Op \psi)\circ g \rvert
\leqslant G^2 \lvert (V\psi) \circ g \rvert + G^2 \lvert \zeta \circ g \rvert
=  \lvert \tilde V \tilde \psi  \rvert +  \lvert \tilde \zeta \rvert .
\]
Let $\tilde A = (\tilde a^{ij})_{i,j=1}^d$ and $\tilde \vartheta_2 = G \vartheta_2$. The coefficients of $\tilde \Op$ satisfy for all $x \in \mathbb{R}^d$ and all $\xi \in \mathbb{R}^d$
\[
 \vartheta_1^{-1} \lvert \xi \rvert^2 \leqslant \xi^\T \tilde A (x) \xi \leqslant \vartheta_1 \lvert \xi \rvert^2, \quad
 \lVert \tilde A (x) - \tilde A (y) \rVert_\infty \leqslant \tilde \vartheta_2 \lvert x-y \rvert ,
\]
$\lVert \tilde b \rVert_\infty \leqslant G \lVert b \rVert_\infty$, $\lVert \tilde c \rVert_\infty \leqslant G^2 \lVert c \rVert_\infty$, and $\lVert \tilde V \rVert_\infty \leqslant G^2 \lVert V \rVert_\infty$. Hence we can apply Theorem~\ref{thm:sampling} to the functions $\tilde \psi$ and $\tilde \zeta$ and obtain
\[
  \lVert \psi \rVert_{S_\delta}^2 = G^d \lVert \tilde \psi \rVert_{S_\delta / G}^2
 \geqslant C_\sfUC G^d \lVert \tilde \psi \rVert_{\mathbb{R}^d}^2 - (\delta / G)^2 G^d \lVert \tilde \zeta \rVert_{\mathbb{R}^d}^2 = C_\sfUC \lVert \psi \rVert_{\mathbb{R}^d}^2 - (\delta / G)^2 G^4\lVert \zeta \rVert_{\mathbb{R}^d}^2 .
\]
\section*{Acknowledgments}

M.T.\ thanks Ivica Naki\'c for stimulating discussions, in particular concerning the smallness condition \eqref{ass:samplingG=1},
and kind hospitality at University of Zagreb where part of this work was done.
This visit was supported by the binational German-Croatian DAAD-MZOS project \emph{Scale-uniform controllability of partial differential equations}.
%
%
%


\begin{thebibliography}{NTTV15}

\bibitem[AM15]{AnantharamanM-15}
N.~Anantharaman and E.~Le Masson.
\newblock Quantum ergodicity on large regular graphs.
\newblock {\em Duke Math. J.}, 164(4):723--765, 2015.

\bibitem[Bak13]{Bakri-13}
L.~Bakri.
\newblock Carleman estimates for the {S}chr\"odinger operator. {A}pplications
  to quantitative uniqueness.
\newblock {\em Commun. Part. Diff. Eq.}, 38(1):69--91, 2013.

\bibitem[BK05]{BourgainK-05}
J.~Bourgain and {C.~E.} Kenig.
\newblock On localization in the continuous {A}nderson-{B}ernoulli model in
  higher dimension.
\newblock {\em Invent. Math.}, 161(2):389--426, 2005.

\bibitem[BK13]{BourgainK-13}
J.~Bourgain and A.~Klein.
\newblock Bounds on the density of states for {S}chr{\"o}dinger operators.
\newblock {\em Invent. Math.}, 194(1):41--72, 2013.

\bibitem[BML16]{BrooksML-15}
S.~Brooks, E.~Le Masson, and E.~Lindenstrauss.
\newblock Quantum ergodicity and averaging operators on the sphere.
\newblock {\em Int. Math. Res. Notices}, 2016(19):6034--6064, 2016.

\bibitem[BSSP03]{BoechererSS-03}
S.~B\"ocherer, P.~Sarnak, and R.~Schulze-Pillot.
\newblock Arithmetic and equidistribution of measures on the sphere.
\newblock {\em Commun. Math. Phys.}, 242(1--2):67--80, 2003.

\bibitem[CHK07]{CombesHK-07}
{J.-M.} Combes, {P. D.} Hislop, and F.~Klopp.
\newblock An optimal {Wegner} estimate and its application to the global
  continuity of the integrated density of states for random {Schr\"odinger}
  operators.
\newblock {\em Duke Math. J.}, 140(3):469--498, 2007.

\bibitem[DF88]{DonnellyF-88}
H.~Donnelly and C.~Fefferman.
\newblock Nodal sets for eigenfunctions on {R}iemannian manifolds.
\newblock {\em Invent. Math.}, 93(1):161--183, 1988.

\bibitem[EV03]{EscauriazaV-03}
L.~Escauriaza and S.~Vessella.
\newblock Optimal three cylinder inequalities for solutions to parabolic
  equations with {L}ipschitz leading coefficients.
\newblock In G.~Alessandrini and G.~Uhlmann, editors, {\em Inverse Problems:
  Theory and Applications}, volume 333 of {\em Contemp. Math.}, pages 79--87.
  American Mathematical Society, Providence, 2003.

\bibitem[GK13]{GerminetK-13}
F.~Germinet and F.~Klopp.
\newblock Enhanced {W}egner and {M}inami estimates and eigenvalue statistics of
  random {A}nderson models at spectral edges.
\newblock {\em Ann. Henri Poincar{\'e}}, 14(5):1263--1285, 2013.

\bibitem[JL99]{JerisonL-99}
D.~Jerison and G.~Lebeau.
\newblock Nodal sets of sums of eigenfunctions.
\newblock In M.~Christ, {C.~E.} Kenig, and C.~Sadosky, editors, {\em Harmonic
  analysis and partial differential equations}, Chicago Lecture notes in
  Mathematics, pages 223--239. The University of Chicago Press, Chicago, 1999.

\bibitem[Kle13]{Klein-13}
A.~Klein.
\newblock Unique continuation principle for spectral projections of
  {S}chr{\"o}dinger operators and optimal {W}egner estimates for non-ergodic
  random {S}chr{\"o}dinger operators.
\newblock {\em Commun. Math. Phys.}, 323(3):1229--1246, 2013.

\bibitem[KSU11]{KenigSU-11}
{C.~E.} Kenig, M.~Salo, and G.~Uhlmann.
\newblock Inverse problems for the anisotropic {M}axwell equations.
\newblock {\em Duke Math. J.}, 157(2):369--419, 2011.

\bibitem[KT16]{KleinT-16}
A.~Klein and {C.~S.~S.} Tsang.
\newblock Quantitative unique continuation principle for {S}chr\"odinger
  operators with singular potentials.
\newblock {\em P. Am. Math. Soc.}, 144(2):665--679, 2016.

\bibitem[Kuk98]{Kukavica-98}
I.~Kukavica.
\newblock Quantitative uniqueness for second-order elliptic operators.
\newblock {\em Duke Math. J.}, 91(2):225--240, 1998.

\bibitem[LL12]{RousseauL-12}
J.~{Le Rousseau} and G.~Lebeau.
\newblock On {C}arleman estimates for elliptic and parabolic operators.
  {A}pplications to unique continuation and control of parabolic equations.
\newblock {\em ESAIM Contr. Op. Ca. Va.}, 18(3):712--747, 2012.

\bibitem[LR95]{LebeauR-95}
G.~Lebeau and L.~Robbiano.
\newblock Contr\^{o}le exact de l'{\'e}quation de la chaleur.
\newblock {\em Commun. Part. Diff. Eq.}, 20(1--2):335--356, 1995.

\bibitem[NRT]{NakicRT-15-arxiv}
I.~Naki\'c, C.~Rose, and M.~Tautenhahn.
\newblock A quantitative {C}arleman estimate for second order elliptic
  operators.
\newblock {\em to appear in P. Roy. Soc. Edinb. A}.
\newblock arXiv:1502.07575 [math.AP].

\bibitem[NTTVa]{NakicTTV-17-prep}
I.~Naki\'c, M.~T\"aufer, M.~Tautenhahn, and I.~Veseli\'c.
\newblock in preparation.

\bibitem[NTTVb]{NakicTTV-16-arxiv}
I.~Naki\'c, M.~T\"aufer, M.~Tautenhahn, and I.~Veseli\'c.
\newblock Scale-free unique continuation principle, eigenvalue lifting and
  {W}egner estimates for random {S}chr\"odinger operators.
\newblock {\em to appear in Anal. PDE}.
\newblock arXiv:1609.01953 [math.AP].

\bibitem[NTTV15]{NakicTTV-15}
I.~Naki\'c, M.~T\"aufer, M.~Tautenhahn, and I.~Veseli\'c.
\newblock Scale-free uncertainty principles and {W}egner estimates for random
  breather potentials.
\newblock {\em C. R. Math.}, 353(10):919--923, 2015.

\bibitem[RMV13]{Rojas-MolinaV-13}
C.~Rojas-Molina and I.~Veseli{\'c}.
\newblock Scale-free unique continuation estimates and applications to random
  {S}chr{\"o}dinger operators.
\newblock {\em Commun. Math. Phys.}, 320(1):245--274, 2013.

\bibitem[TT17]{TaeuferT-17}
M.~T\"aufer and M.~Tautenhahn.
\newblock Scale-free and quantitative unique continuation for infinite
  dimensional spectral subspaces of {S}chr{\"o}dinger operators.
\newblock {\em Commun. Pur. Appl. Anal.}, 16(5):1719--1730, 2017.

\bibitem[TTV16]{TaeuferTV-16}
M.~T{\"a}ufer, M.~Tautenhahn, and I.~Veseli\'c.
\newblock Harmonic analysis and random schr\"odinger operators.
\newblock In M.~Mantoiu, G.~Raikov, and R.~{Tiedra de Aldecoa}, editors, {\em
  Spectral Theory and Mathematical Physics}, volume 254 of {\em Operator
  Theory: Advances and Applications}, pages 223--255. Birkh\"auser, Basel,
  2016.

\bibitem[TV15]{TautenhahnV-15}
M.~Tautenhahn and I.~Veseli\'c.
\newblock Discrete alloy-type models: Regularity of distributions and recent
  results.
\newblock {\em Markov Process. Related Fields}, 21(3):823--846, 2015.

\bibitem[Zel92]{Zelditch-92}
S.~Zelditch.
\newblock Quantum ergodicity on the sphere.
\newblock {\em Commun. Math. Phys.}, 146(1):61--71, 1992.

\end{thebibliography}
\end{document}